\documentclass{amsproc}
\usepackage[margin=1.2in,nomarginpar]{geometry}
\usepackage{amssymb}
\usepackage{amsmath}
\usepackage{mathdots}
\usepackage{amsbsy}
\usepackage{amscd}
\usepackage{amsthm}

\usepackage{url}
\usepackage{xcolor}
\usepackage{amsmath,amssymb,amsthm,amsfonts,longtable}
\usepackage[english]{babel}
\usepackage{tikz-cd}
\usetikzlibrary{cd}
\usetikzlibrary{decorations.markings}
\tikzset{negated/.style={
    decoration={markings,
      mark= at position 0.5 with {
        \node[transform shape] (tempnode) {$\times$};
      }
    },
    postaction={decorate}
  }
}
\usepackage{array}
\usepackage[colorlinks,citecolor=red,urlcolor=blue,bookmarks=false,hypertexnames=true]{hyperref} 
\newtheorem{theorem}{Theorem}
\newtheorem{corollary}[theorem]{Corollary}
\newtheorem{lemma}[theorem]{Lemma}
\newtheorem{setup}[theorem]{Set up}
\newtheorem{remark}[theorem]{Remark}
\newtheorem{example}[theorem]{Example}
\newtheorem{definition}{Definition}[subsection]
\newcommand{\Irr}{\textnormal{Irr}}
\newcommand{\cd}{\textnormal{cd}}
\newcommand{\nl}{\textnormal{nl}}
\newcommand{\lin}{\textnormal{lin}}
\newcommand{\gal}{\textnormal{Gal}}
\newcommand{\Mod}[1]{\ (\mathrm{mod}\ #1)}

\title[Faithful quasi-permutation representations]{On faithful quasi-permutation representations of\\ VZ-groups and Camina $p$-groups}
\author{Sunil Kumar Prajapati$^*$}
\address{Indian Institute of Technology, Bhubaneswar, Arugul Campus, Jatni, Khurda-752050, India.}
\email{skprajapati@iitbbs.ac.in}
\author{Ayush Udeep}
\address{Indian Institute of Technology, Bhubaneswar, Arugul Campus, Jatni, Khurda-752050, India.}
\email{udeepayush@gmail.com}
\thanks{$^{\textbf{*}}$ Corresponding author.
}
\begin{document}
\maketitle

%\begin{document}
	%\setcounter{page}{1} 
	%\noindent \\ 
	%ISSN 1310--5132 \\
	%Vol. XX, XXXX, No. X, XX--XX}
	%\vspace{10mm}

%	\begin{center}
%		{\Large \bf On faithful quasi-permutation representations of $VZ$-groups and Camina $p$-groups}
%		\vspace{8mm}
%		
%		%{Author 1 and Author 2  }
%		%\vspace{3mm}
%		
%		
%		
%		
%	\end{center}
%	\vspace{10mm}
	
\begin{abstract}
For a finite group $G$, we denote by $\mu(G)$ and $c(G),$ the minimal degree of faithful permutation representation of $G$ and the minimal degree of faithful representation of $G$ by quasi-permutation matrices over the complex field $\mathbb{C}$. In this paper, we examine $c(G)$ and $\mu(G)$ for $VZ$-groups and Camina $p$-groups. 
\end{abstract}
	
	\section{Introduction}
	Throughout this paper, $G$ is a finite group and $\Irr(G)$ is the set of irreducible complex characters of $G$. By a classical Cayley's theorem, $G$ can be embedded into the Symmetric group $S_{|G|}$. It is interesting to search for smallest $n$ such that $G$ is contained in $S_n$. The minimal faithful permutation degree $\mu(G)$ of $G$ is the least positive integer $n$ such that $G$ embeds inside $S_{n}$. As we know, every permutation representation can be expressed as
$\mathcal{P} =\{H_1,\ldots, H_n\}$ and $\deg(\mathcal{P}) =\sum_{i=1}^n |G / H_i |$, where $H_i$'s are subgroups of $G$. So we may restate the minimal faithful degree of a group $G$ as
\[ \mu(G)= \min \left\{ \sum\limits_{i=1}^{n}|G/H_{i}|~|~ H_{i}\leq G \text{ for } i=1,2,\ldots,n \text{ and } \bigcap_{\substack{i=1}}^{n}\bigcap_{\substack{x\in G}}{H_{i}}^{x}=1  \right\} \] (see \cite{HB}).
In a parallel direction to the definition of permutation group, in 1963, Wong \cite{W}  defined a quasi-permutation group in the following way: If $G$ is a finite linear group of degree $n$ (i.e. a finite group of automorphisms of an $n$-dimensional complex vector space) such that the trace of every element of $G$ is a non-negative integer, then $G$ is called a quasi-permutation group. 
	This terminology allows us to have another degree $c(G)$, denoting the minimal degree of a faithful representation of $G$ by complex quasi-permutation matrices. Since every permutation matrix is a quasi-permutation matrix, it is easy to see that $c(G)\leq \mu(G)$.  Another interesting quantity is $q(G)$, which denotes the  minimal degree of a faithful representation
of G by quasi-permutation matrices over the rational field $\mathbb{Q}$. In past, several researchers have carried out the study of $c(G), q(G)$ and $\mu(G)$. 
The interested reader may refer to 
\cite{AB, BD, BG, HB1997, HB, BGHS, GA, MG, W, DW}. It is easy to see that, for a finite group $G$, we have the following relation between $\mu(G), ~ c(G)$ and $q(G):$  
$$c(G)\leq q(G) \leq \mu(G).$$

 In this article, we study $c(G)$ and $\mu(G)$ for two types of finite non-abelian groups, type 1: group $G$ for which the set 
 $\cd(G)=\{\chi(1)~|~ \chi \in \Irr(G)\}= \{1, |G/Z(G)|^{1/2}\}$ and type 2: Camina $p$-groups. Groups of type 1 are also known as VZ-group and their nilpotency class is $2$. A group is called a VZ-group if all its nonlinear irreducible complex characters vanish off the center. VZ-group was studied in \cite{FAM, L1, L2009}. A $p$-group $G$ is called a Camina $p$-group if the conjugacy class of $g$ is $gG{}'$ for all $g\in G\setminus G{}'$.
Study of Camina groups was influenced by investigating common generalization of Frobenius groups
and extra-special groups, and it has been studied by several researchers (\cite{DS, IL, lewis}).

In \cite{BG}, authors gave an algorithm for the computation of $c(G)$. They proved that if $X$ is a subset of $\Irr(G)$ such that $\cap_{\chi \in X} \ker (\chi)= 1$, and for any proper subset $Y$ of $X,~ \cap_{\chi \in Y} \ker (\chi) \neq 1,$ then $c(G) \leq \xi(1) + m(\xi),$ where $ \xi = \sum_{\chi \in X} \left[ \sum_{\sigma \in \Gamma(\chi)} \chi^{\sigma}  \right], $
 where $\Gamma(\chi)$ is the Galois group of $\mathbb{Q}(\chi)$ over $\mathbb{Q}$ and $m(\xi)$ is the absolute value of the least value of $\sum_{\chi \in X}\sum_{\sigma \in \Gamma(\chi) } \chi^{\sigma}$ over $G$ (see \cite[Lemma 2.2]{BG}). 
We call a set $X_G\subset \Irr(G)$, a minimal faithful quasi-permutation representation of $G$ if \begin{equation}\label{eq:X_G} 
		\bigcap_{\chi \in X_G} \ker (\chi) = 1,
		\end{equation}
		such that equation \eqref{eq:X_G} does not hold for any proper subset of $X_G$ and $c(G) = \xi(1) + m(\xi)$. 

%\noindent	The aim of this article is to compute $c(G)$ for some fairly new classes of groups. To describe those classes, we recall some definitions. Let $N$ be a normal subgroup of $G$. We say that $(G, N)$ is a Camina pair if every element $g \in G \setminus N$ is conjugate to every element of $gN.$ When $N=G'$, $G$ is called a Camina group.  These pairs were first studied by Camina \cite{C1} and later by various researchers such as Macdonald \cite{M1,M2}, Fern\'andez-Alcober and Moret\'o \cite{FAM}, and Dark and Scoppola \cite{DS}. Lewis \cite{L1} defined a similar pair called a  generalized Camina Pair (abbreviated GCP) in the following way. A pair $(G, N)$ is said to be a GCP if $N$ is normal in $G$, and all non-linear irreducible characters of $G$ vanish outside $N.$  When $(G, Z(G))$ forms a GCP, then $G$ is called a $VZ-$group \cite{L2}.\\
%	In this article, we have investigated $c(G)$ and $\mu(G)$ for $VZ$ groups and Camina $p$-groups. 
	In \cite{BG}, authors proved that for a finite $p$-group $G$ with $p$ odd prime, $q(G) = \mu(G)$.
Moreover, for odd prime $p$, $c(G) = q(G) = \mu(G)$. In \cite[Theorem 4.12]{HB}, author proved that if $G$ is a finite $p$-group of class $2$ and $Z(G)$ is cyclic, then 
$c(G)= |G/Z(G)|^{1/2}|Z(G)|$. Note that a finite $p$-group of class $2$ with cyclic center may or may not be a VZ $p$-group. Suppose $\nl(G)$ and $\lin(G)$ denote the set of all non-linear irreducible complex characters of $G$ and the set of all linear characters of $G$, respectively. In this article, to prove our results, we use the following set up.

\begin{setup}\label{setup:thm1} \textnormal{ \noindent Suppose $d(G)$ denotes the minimal number of generators of $G$. Let $G$ be a VZ $p$-group, $d(Z(G)) = r$ and $d(G') = k.$ Here $G' \leq Z(G)$ and  $G{}'$ is elementary abelian $p$-group (by Lemma \ref{lem:VZgroup}).
Then we can choose $a_{1}, a_{2}, \ldots, a_{r}\in G$ such that
\begin{equation} \label{eq:theorem1}
Z(G) = \langle a_{1},a_{2},\ldots, a_{r} \rangle=H_{1}\times \cdots \times H_{r} \cong C_{p^{l_{1}}} \times C_{p^{l_{2}}} \times \cdots \times C_{p^{l_{r}}}, ~ G' = \langle a_{1}^{p^{l_{1}-1}}, \ldots, a_{k}^{p^{l_{k}-1}} \rangle,
\end{equation}
where $H_{i}=\langle a_{i} \rangle$ 
and 
\begin{equation} \label{eq:theorem20}
\frac{Z(G)}{G'} = \langle a_{1}G',\ldots,  a_{k}G', a_{k+1}G',\ldots,a_{r}G' \rangle \cong C_{p^{l_{1}-1}} \times \cdots \times C_{p^{l_{k}-1}} \times C_{p^{l_{k+1}}} \cdots \times C_{p^{l_{r}}}.
\end{equation}
Set, for each $i$  ($1\leq i \leq r$) $\widehat{H_{i}}=H_{1}\times \cdots \times H_{i-1}\times H_{i+1}\times \cdots \times H_{r}$.
}
\end{setup}
\noindent Under the above set up, we have the following theorem. 
	\begin{theorem} \label{T1}
		Let $G$ be a VZ $p$-group, for any prime $p$. Suppose $d(Z(G)) = r$ and  $d(G')= k$. For each $i$ ($1\leq i\leq k$), 
let $\chi_{i} \in \nl(G)$ be such that $\ker (\chi_{i}) = \widehat{H_{i}}$ and for $k+1\leq  i\leq r,$ let $\psi_{i} \in \lin(G)$  be such that $\ker (\psi_{i}) \cap Z(G) = \widehat{H_{i}}$.  Then $X_G = \{ \chi_{1},\ldots,\chi_{k},\psi_{k+1},\ldots,\psi_{r} \}$ is a minimal faithful quasi-permutation representation of $G$. Further, suppose $|G/\ker(\psi_{i})|=p^{m_{i}}$ for $k+1\leq  i\leq r$.  Then
		\begin{equation} \label{eq:25}
		c(G) = |G/Z(G)|^{1/2} \sum\limits_{i=1}^{k} |H_{i}| + \sum\limits_{i=k+1}^{r} p^{m_{i}}.
		\end{equation}
If $p$ is an odd prime, then $c(G)=\mu(G)=q(G)=|G/Z(G)|^{1/2} \sum\limits_{i=1}^{k} |H_{i}| + \sum\limits_{i=k+1}^{r} p^{m_{i}}$. 
	\end{theorem}
	
\noindent For VZ-groups, in \cite{AB} authors proved the following result. 
\begin{lemma}\cite[Lemma 2.2]{AB}\label{lem:finteVZgroup}
If $G$ is a finite group with $G{}'=Z(G)$, and $\cd(G)=\{1, |G/Z(G)|^{1/2}\}$, then $c(G)= |G/Z(G)|^{1/2}c(Z(G))$.
\end{lemma}
\noindent Note that a finite VZ-group is a nilpotent group of class $2$ and thus can be written as a direct product of its Sylow subgroups. Since $\cd(G)=\{1, |G/Z(G)|^{1/2}\}$, all the Sylow subgroups of $G$, except one, are abelian. 
Now by this observation, group mentioned in Lemma \ref{lem:finteVZgroup} must be a finite $p$-group. 
From Theorem \ref{T1}, we deduce the following result, which is a generalization of Lemma \ref{lem:finteVZgroup}.
\begin{corollary} \label{C2}
	Let $G$ be a VZ $p$-group (prime $p$) such that $d(G') = d(Z(G))$. Then
	\[ c(G) = |G/Z(G)|^{1/2} c(Z(G)). \]
\end{corollary}
\noindent Since VZ-group is a nilpotent group of class $2$, we deduce the following result for VZ-group of odd order.  
\begin{corollary} \label{C3}
	Let $G$ be a VZ group of odd order. Then $c(G)=\mu(G)=q(G)$. 
\end{corollary} 	

Further, we examine $c(G)$ for the family of Camina $p$-groups. In \cite{DS}, authors proved that if $G$ is a
 finite Camina $p$-group, then the nilpotency class of $G$ is at most 3. It is easy to see that if $G$ is Camina $p$-group of nilpotency class 2, then $G$ is indeed a VZ-group with $Z(G)=G{}'$ and hence by Corollary \ref{C2}, we get $c(G)=|G/Z(G)|^{1/2}c(Z(G))$.  For Camina $p$-group of nilpotency class 3, we have proved the following result. 
		
%	\begin{corollary} \label{C1}
%		Let $G$ be a Camina $p$-group of nilpotency class 2. Then 
%		\[ c(G) = |G:Z(G)|^{1/2} c(Z(G)). \]
%	\end{corollary}
%	\noindent \emph{Proof.} 
%	Let $G$ be a Camina $p$-group of nilpotency class 2. Then, from Lemma 2.1 of \cite{M1} $G$ is a $VZ$-group, and from Corollary 2.4 of \cite{M1}, $G'=Z(G).$ Thus, $d(G') = d(Z(G))$ and then from Corollary \ref{C2}, the result follows.
	
		\begin{theorem} \label{T2}
		Let $G$ be a Camina $p$-group of nilpotency class $3$ (prime $p$). Then
		\begin{equation} \label{eq:14}
		c(G) = |G/Z(G)|^{1/2}c(Z(G)).
		\end{equation} 
Moreover, when $p$ is an odd prime, we get $c(G) =\mu(G)=q(G)=|G/Z(G)|^{1/2}c(Z(G)).$
	\end{theorem}
%	\noindent The main theorem of Dark and Scoppola (\cite{DS}, p. 788) tells us that the nilpotency class of a finite Camina $p$-group is at most 3, and thus Corollary \ref{C1} and Theorem \ref{T2} deals with all finite Camina $p$-groups.
In Section \ref{section:BR}, we briefly summarize the background results that will be used in
the course of this paper. In Section \ref{section:VZ}, we first try to find out $c(G)$ for VZ $p$-group by explicit description of the set $X_G\subseteq \Irr(G)$. Since VZ-groups are nilpotent groups, we get $c(G)=\mu(G)$, for VZ-groups of odd order. In Section \ref{sec:isoclexam}, by using the concept of isoclinism, we have shown that: 
If $G$ and $H$ are two isoclinic VZ $p$-groups of same order such that $d(Z(G)) = d(G{}')$, then $c(G) = c(H)$ if and only if $Z(G)$ is isomorphic to $Z(H)$ (Corollary \ref{extra1}). By using Corollary \ref{extra1} and Theorem \ref{T1}, we have computed $c(G)$ for all the $p$-groups of size $p^6$ (for an odd prime $p$)(see Table \ref{t:1}, Table \ref{t:2} and Corollary \ref{extra2}).
In the last section, we have calculated $c(G)$ for Camina $p$-group and provided $X_G$. 
	\section{Basic results on $c(G),~ q(G)$ and $\mu(G)$}\label{section:BR}
 Suppose $F$ is a subfield of $E$, where $E$ is a splitting field for $G$. Let $\chi, ~\psi \in \Irr_E(G)=\Irr(G)$. 
We say that $\chi$ and $\psi$ are Galois conjugate over $F$ if $F(\chi)=F(\psi)$ and there exists $\sigma \in \gal(F(\chi)/F)$ such that $\chi^{\sigma}= \psi$,
where $F(\chi)$ denotes the field obtained by adjoining  the values $\chi(g)$, for all
$g \in G$, to $F$. Indeed, this defines an equivalence relation on $\Irr_E(G)$. Moreover, if $\mathcal{C}$ denotes the equivalence class of $\chi$ with respect to Galois conjugacy over $F$, then $|\mathcal{C}|=| F(\chi) : F |$ (see \cite[Lemma 9.17]{I}).	
	
 Suppose $\chi$ is an irreducible  character of a group $G$ and $\Gamma(\chi)$ denotes the Galois group of $\mathbb{Q}(\chi)$ over $\mathbb{Q}$. 
 Let $\mathcal{C}_{i}$, for $0\leq i\leq r$, be the Galois conjugacy classes of irreducible characters of the group $G$. For $0\leq i \leq r$, let $\psi_{i}$ be a representative of the class $\mathcal{C}_{i}$ with $\psi_{0}=1_{G}$ (trivial character). Let us write $\Theta_{i}= \sum_{\chi\in \mathcal{C}_{i}} \chi$,  and $K_{i}=\ker (\psi_{i}).$ For $I\subseteq \{ 0, 1, \ldots, r \}$, set $K_{I}=\cap_{i\in I}K_{i}.$ Under the above set up, we have following definition and results which we use to calculate $c(G)$. 
	\begin{definition} \label{D1}
		Let $G$ be a finite group. Let $\chi$ be an irreducible complex character of $G$. Then define 
		\begin{enumerate}
			\item [\rmfamily(i)] $d(\chi)=|\Gamma(\chi)|\chi(1)$.
			\item [\rmfamily(ii)] $ m(\chi)=
			\begin{cases}
			0 &\quad \text{ if } \chi= 1_{G},\\
			\left|\min \left\{ \sum\limits_{\sigma\in \Gamma(\chi)} \chi^{\sigma}(g): g\in G \right\}\right| &\quad \text{ otherwise.} 
			\end{cases}$
%			\item [\rmfamily(iii)] $c(\chi)= \sum\limits_{\sigma\in \Gamma(\chi)}\chi^{\sigma}+ m(\chi)1_{G}$, where $1_{G}$ denotes the trivial character of $G.$
		\end{enumerate}
	\end{definition}
	
	\begin{lemma}\textnormal{\cite[Lemma 2.2]{BG}}\label{L5}
		Let $G$ be a finite group. Then 
		\begin{align*}
		c(G)= \min \Bigg\{ &\xi(1)+ m(\xi): \xi=\sum\limits_{i\in I}\Theta_{i}, K_{I}=1, \textnormal{for}~ I\subseteq \{ 1,\ldots,r \}~  
		\textnormal{and}~ K_{J}\neq 1, \textnormal{ if } J\subset I \Bigg\}
		\end{align*}
and 		
\begin{align*}
		q(G)= \min \Bigg\{ &\xi(1)+ m(\xi): \xi=\sum\limits_{i\in I}m_{\mathbb{Q}}(\psi_i)\Theta_{i}, K_{I}=1, \textnormal{for}~ I\subseteq \{ 1,\ldots, r \} 
		\textnormal{and}~ K_{J}\neq 1, \textnormal{ if } J\subset I \Bigg\}, 
		\end{align*}			
where 	$m_{\mathbb{Q}}(\psi_i)$ denotes the Schur index of $\psi_i$ over $\mathbb{Q}$. 	
	\end{lemma}
If $G$ is a $p$-group ($p$ odd prime), then by Lemma \ref{L5}, $c(G)=q(G)$ because $m_{\mathbb{Q}}(\psi)=1$ for every $\psi \in \Irr(G)$. 
% For $2$-groups, we have the following result.
%\begin{lemma} \textnormal{\cite[Lemma 3.1]{BG2006}} \label{schurindex}
%	Let $G$ be a $2$-group and $\chi \in \Irr(G)$. Then $m_{\mathbb{Q}}(\chi) = m_{\mathbb{R}}(\chi)$.
%\end{lemma}
%\noindent From the theory of Schur index, we know that $m_{\mathbb{R}}(\chi) = 1$, when $\chi$ is not real (see \cite{RG}). 
	\begin{lemma}\textnormal{\cite[ Corollary 2.6]{BD}} \label{L1}
		Let $G$ be a $p$-group whose center $Z(G)$ is minimally generated by $d$ elements. Let $c(G) = \xi(1)+m(\xi)$ and $\xi = \sum_{i\in I} \Psi_{i}.$ Let $\Psi_{i}'s$ satisfy the conditions of the algorithm for $c(G).$ Then
		\begin{enumerate}
			\item [\rmfamily(i)] $m(\xi) = \frac{1}{p-1} \sum\limits_{i\in I} \Psi_{i},$
			\item [\rmfamily(ii)] $|I| = d.$
		\end{enumerate}
	\end{lemma}
Suppose $G$ is an abelian group isomorphic to $\prod_{i=1}^k C_{p_i^{r_i}}$, where $C_{p_i^{r_i}}$'s are cyclic groups of prime power order $p_i^{r_i}$. Define $T(G)=\sum_{i=1}^k p_i^{r_i}$. In \cite{HB1997}, author proved that $c(G)=T(G)-n$, where $n$ is the largest integer such that, $C_6^n$ is a direct summand of an abelian group $G$. Further, if $G$ is an abelian $p$-group, then $c(G)=q(G)=\mu(G)=T(G)$. 

\begin{theorem}\textnormal{\cite[Corollary 2.2]{DW}}\label{thm:nilpotent} If $H, K$ are  non-trivial nilpotent groups, then $\mu(H \times K)= \mu(H)+ \mu(K)$.
\end{theorem}

\begin{theorem}\textnormal{\cite[Theorem 3.2]{BG}}
\label{thm:pgroup}
Let $G$ be a finite $p$-group. Then $q(G) = \mu(G)$. Moreover, if $p\neq 2$, then $c(G) = q(G) = \mu(G)$.
\end{theorem}

\begin{theorem}\textnormal{\cite[Theorem 4.12]{HB}}
\label{thm:cycliccenter}
Let G be a finite $p$-group of class 2 and let $Z(G)$ be cyclic. Then
$c(G)=c(Z(G))|G/Z(G)|^{1/2}$.
\end{theorem}

	\section{VZ group}\label{section:VZ}
	
Suppose $G$ is a VZ-group. Then  $\chi(g) = 0$, for all $g \in G \setminus Z(G)$ and for all $\chi\in \nl(G)$, if and only if 
$\cd(G)=\{1, |G/Z(G)|^{1/2} \}$. Moreover, $|\nl(G)|=|Z(G)|-\frac{|Z(G)|}{|G{}'|}$ and $\nl(G)=\{|G/Z(G)|^{1/2}\lambda ~|~ \lambda \in \Irr(Z(G)) \textnormal{ such that } G{}' \nleq \ker(\lambda) \}$, where members of $\nl(G)$ are defined as follows:
$$|G/Z(G)|^{1/2}\lambda (g):=\begin{cases}
			|G/Z(G)|^{1/2}\lambda(g) &\quad \text{ if } g\in Z(G),\\
			0 &\quad \text{ otherwise} 
			\end{cases} $$
(see \cite[Subsection 3.1]{SR}). Obviously, $\ker(\chi) \lneq Z(G)$ for all $\chi\in \nl(G)$. 

\begin{lemma}\label{lem:ramified} Suppose $G$ is a VZ group. Then every $\chi\in \nl(G)$ is fully ramified with respect to $G/Z(G)$. 
\end{lemma}
\begin{proof} Follows from Problem $6.3$ of \cite{I}.
\end{proof}

\begin{lemma}\label{lem:galoisconjugate} Suppose $G$ is a VZ group. Let $\chi, \psi\in \nl(G)$. Then $\chi$ and $\psi$ are Galois conjugate if and only if $\ker(\chi)=\ker(\psi)$.
\end{lemma}
\begin{proof} It follows from the fact that two linear characters of an abelian group are Galois conjugate if and only if their kernels are same. 
\end{proof}

Now we quote the group-theoretic information of $G/Z(G)$ and $G{}'$ for a VZ $p$-group. 
\begin{lemma}\textnormal{\cite[Lemma 2.4]{L2}}\label{lem:VZgroup} Let $G$ be a VZ $p$-group. Then $G/Z(G)$ and $G{}'$ are elementary abelian
$p$-groups. 
\end{lemma}	

\begin{lemma}\label{lem:nonlinearkernel2} Suppose $G$ is a VZ $p$-group and $d(G{}')\geq 2$. Then for all $\chi\in \nl(G)$, $\ker(\chi)$ intersects with $G{}'$ non-trivially. 
\end{lemma} 
\begin{proof} By the above observation, if $\chi \in \nl(G)$, then $\chi=|G/Z(G)|^{1/2}\lambda$ for some $\lambda \in \Irr(Z(G))$ such that $G{}' \nleq \ker(\lambda)$ and $\ker(\chi)=\ker(\lambda) \lneq Z(G)$. On the contrary, suppose $\ker(\chi) \cap G{}' ={1}$. Consider the natural group homomorphism $\eta : Z(G) \longrightarrow Z(G)/\ker(\chi)$ such that $g \mapsto g \ker(\chi)$. Take the restriction  $\eta\downarrow _{G{}'} : G{}' \longrightarrow Z(G)/\ker(\chi)$. Under the assumption $\ker(\chi) \cap G{}' ={1}$,  $\eta\downarrow _{G{}'} $ is injective group homomorphism and thus $G{}'$	 is isomorphic to a subgroup of the cyclic group $Z(G)/\ker(\chi)$. This is a contradiction because $d(G{}')\geq 2$.  
\end{proof}	
%\noindent By Lemma \ref{lem:nonlinearkernel2}, it is easy to see the following general result. 
% 
% \begin{lemma}\label{lem:VZnormal} Suppose $G$ is a VZ $p$-group and $d(G{}')\geq 2$. Then if $N$ is a proper normal subgroup of $G$, then either $G{}' \leq N$ or $N \leq Z(G)$ with $N \cap G{}' \neq \{1\}$. 
% \end{lemma}
		
% Also from  Lemma 2.27 of \cite{I}, $\chi_{Z(G)} = \chi(1) \lambda,$ where $\lambda$ is a linear character of the center $Z(G)$ . Thus \\
%	\begin{align*}
%	|G| = \sum\limits_{g\in G} |\chi(g)|^{2} =& \sum\limits_{g\in Z(G)}|\chi(g)|^{2}  \quad \text{ (since $G$ is a $VZ$-group)} \\
%	=&  \sum\limits_{g\in Z(G)}|\chi(1)\lambda(g)|^{2} 
%	= \chi(1)^{2} |Z(G)|.
%	\end{align*}
%	Therefore, the degree of any non-linear irreducible character of $G$ is $|G:Z(G)|^{1/2},$ and thus $cd(G) = \{ 1, |G:Z(G)|^{1/2} \}.$ \\
%	
\begin{lemma}\label{lem:nonlinearkernel} Suppose $G$ is a VZ $p$-group. Let $G{}'=\langle g_1, g_2, \ldots, g_k \rangle$. Then for any  $\chi\in \nl(G)$, there exists $g_i$ such that $\langle g_i\ker(\chi)\rangle$ is a subgroup of cyclic group $Z(G)/\ker(\chi)$. 
\end{lemma}
\begin{proof}
Suppose $G{}'=\langle g_1, g_2, \ldots, g_k \rangle \cong C_p\times C_p\times \cdots \times C_p$ (k copy of $C_p$). Let $\chi\in \nl(G)$. Then there exists $g\in G{}'$ such that $g=g_1^{m_1}g_2^{m_2}\cdots g_k^{m_k}$ for some $m_i$ ($1\leq m_i\leq p$ and $1\leq i\leq k$) and $$\chi(g)=\chi(1)\lambda(g_1^{m_1}g_2^{m_2}\cdots g_k^{m_k})=\chi(1)\lambda(g_1^{m_1})\lambda(g_2^{m_2})\cdots \lambda(g_k^{m_k})\neq \chi(1)
,$$ where $\lambda\in \Irr(Z(G))$ such that $G{}' \nleq \ker(\lambda)$. This implies that there exists $g_i^{m_i}$ such that $\lambda(g_i^{m_i})\neq 1$, i.e. $\chi(g_i^{m_i})\neq \chi(1)$ for some $i$. Since order of $g_i$ is $p$, we get $\chi(g_i)\neq \chi(1)$. Since $Z(G)/\ker(\chi)$ is cyclic group, $\langle g_i\ker(\chi)\rangle \leq Z(G)/\ker(\chi)$.
\end{proof}

	\begin{lemma} \label{L6}
		Let $G$ be a VZ $p$-group with $d(Z(G)) = r$ and $d(G') = k.$ Let $X$ be a set of $r$ irreducible complex characters of $G$ such that 
		\begin{equation*} \label{eq:12}
		\bigcap_{\chi \in X} \ker (\chi) = 1.
		\end{equation*}
		Then the number of linear characters in $X$ must be $r-k$. In fact, the set $X_G$, which represents a minimal faithful quasi-permutation representation of $G$, must contain $k$ many non-linear irreducible characters and $r-k$ many linear characters of $G$.
	\end{lemma}
\begin{proof} We will prove this lemma in two steps. In the first step, we will show that there are at least  $k$ many non-linear irreducible characters in $X$. In the second step, we will show that the number of non-linear irreducible characters in the set $X$ is equal to $k$. \\
{\bf Step 1:}  We will prove this part in cases.\\
{\bf Case 1 ($d(G{}')=1$, i.e. $G{}' \cong C_{p}$).}  Since all the irreducible characters in the set $X$ can not be linear, we can take at most $r-k=r-1$ linear characters in $X$.\\
%  If $Z(G)$ is cyclic then, $G$ has a faithful non-linear irreducible character say $\chi$ and thus $X=\{\chi\}$. Thus the statement of the lemma holds. Now consider $Z(G)$ is non-cyclic, i.e. $d(Z(G))\geq 2$. In this situation we get a non-linear irreducible character say $\chi$ such that 
%$G{}'\cap \ker(\chi)=\{1\}$. Since all the irreducoible characters in $X$ can't be linear, we can keep $\chi$ in $X$. Thus again statement holds. 
{\bf Case 2 ($d(G{}')>1$).} Let $d(Z(G)) = r$ and let $d(G{}') = k$. On the contrary, suppose in $X$, we have only $k-1$ many non-linear irreducible characters of $G$, say $\chi_1, \chi_2, \ldots, \chi_{k-1}$. Since $d(G{}')=k$ and $G{}'$ is an elementary abelian $p$-group, $G{}'$ is isomorphic to direct product of $k$ many copies of cyclic group of order $p$. For each $i$, since $Z(G)/\ker(\chi_i)$ is cyclic, $\ker(\chi_i)$ contains all cyclic components of $G{}'$ except one. Hence $G{}' \cap (\cap_{i=1}^{k-1}\ker(\chi_i))\neq 1$. Further, $G{}'$ is contained in each $\ker(\eta)$, where $\eta \in \lin(G)$. This shows that $\bigcap_{\chi \in X} \ker (\chi) \neq 1,$ which is a contradiction. This proves that in the set $X$, we require at least $k$ many non-linear irreducible characters.\\
{\bf Step 2:} 
Now we will show that the number of non-linear irreducible characters in the set $X$ must be $k$. In the view of Set Up \ref{setup:thm1}, we have $a_{1}, a_{2}, \ldots, a_{r}\in G$ such that
\begin{equation*} 
Z(G) = \langle a_{1},a_{2},\ldots, a_{r} \rangle=H_{1}\times \cdots \times H_{r} \cong C_{p^{l_{1}}} \times C_{p^{l_{2}}} \times \cdots \times C_{p^{l_{r}}}, ~ G{}' = \langle a_{1}^{p^{l_{1}-1}}, \ldots, a_{k}^{p^{l_{k}-1}} \rangle,
\end{equation*}
where $H_{i}=\langle a_{i} \rangle$ 
and 
\begin{equation*} 
\frac{Z(G)}{G{}'} = \langle a_{1}G{}',\ldots,  a_{k}G{}', a_{k+1}G{}',\ldots,a_{r}G{}' \rangle \cong C_{p^{l_{1}-1}} \times \cdots \times C_{p^{l_{k}-1}} \times C_{p^{l_{k+1}}} \cdots \times C_{p^{l_{r}}}.
\end{equation*}
By Lemma \ref{lem:nonlinearkernel}, without loss of generality, we can choose $k$ many non-linear irreducible characters $\chi_1,\ldots,\chi_k$ such that  
$\langle a_i^{p^{l_i-1}}\ker(\chi_i)\rangle \leq Z(G)/\ker(\chi_i)$, for $1\leq i\leq k$. This shows that, for each $i$, we have $\langle a_i\ker(\chi_i)\rangle \leq Z(G)/\ker(\chi_i)$ and no power of $a_i$ belongs to $\ker(\chi_i)$. Since $|\langle a_i\ker(\chi_i)\rangle |=|\langle a_i\rangle |$, we have $Z(G)/\ker(\chi_i)\cong \langle a_i\rangle$ for $1\leq i\leq k$. Now suppose if we add one more non-linear irreducible character, say $\chi_{k+1}$, in the set $X$, then by Lemma \ref{lem:nonlinearkernel}, we get $\langle a_i^{p^{l_i-1}}\ker(\chi_{k+1})\rangle \leq Z(G)/\ker(\chi_{k+1})$, for some fixed $i$, $1\leq i\leq k$. This shows that $\langle a_i\rangle \cong Z(G)/\ker(\chi_{i}) \cong Z(G)/\ker(\chi_{k+1})$ and hence $\chi_i$ and $\chi_{k+1}$ are Galois conjugate. Thus it is meaningless to keep $\chi_{k+1}$ in the set $X$. Therefore, in the set $X$, the number of non-linear characters is equal to $k$ and the rest $r-k$ characters in $X$ are linear characters.\\
Now, from Lemma \ref{L1}, we get $|X_G| = r$, and from equation \eqref{eq:X_G}, $\cap_{\chi \in X_{G}} \ker (\chi) = 1$. Therefore, $X_G$ contains $k$ many non-linear irreducible characters and $r-k$ many linear characters of $G$.
\end{proof}

	\subsection{Proof of Theorem \ref{T1}}

\noindent Let $G$ be a VZ $p-$group, $d(Z(G)) = r$ and $d(G') = k.$ Here $G' \leq Z(G)$ and by Lemma \ref{lem:VZgroup}, $G{}'$ is elementary abelian $p$-group. In the view of Set up \ref{setup:thm1}, we have
\begin{equation} \label{eq:1}
Z(G) = \langle a_{1},a_{2},\ldots, a_{r} \rangle \cong C_{p^{l_{1}}} \times C_{p^{l_{2}}} \times \cdots \times C_{p^{l_{r}}}, ~ G' = \langle a_{1}^{p^{l_{1}-1}}, \ldots, a_{k}^{p^{l_{k}-1}} \rangle
\end{equation}
\begin{equation*} \label{eq:20}
\text{ and } \frac{Z(G)}{G'} = \langle a_{1}G',\ldots,  a_{k}G', a_{k+1}G',\ldots,a_{r}G' \rangle \cong C_{p^{l_{1}-1}} \times \cdots \times C_{p^{l_{k}-1}} \times C_{p^{l_{k+1}}} \cdots \times C_{p^{l_{r}}}.
\end{equation*}

\noindent \emph{\bf Step I: To get a set ${X}$ of $r$ irreducible characters which satisfies equation \eqref{eq:X_G}. }\\ 
For all $i$ ($1\leq i \leq r)$, let $K_{i}= \langle a_{i}G'\rangle  \leq Z(G)/G'$ and let $\widehat{K_{i}} = K_{1} \times \cdots \times K_{i-1}\times K_{i+1} \times \cdots \times K_{r}$. 
Choose $\psi_{k+1},\ldots,\psi_{r} \in \lin(G/G{}')$ such that $\ker(\psi_{i}\downarrow_{Z(G)/G{}'}) = \widehat{K_i}$ for all  $i$ ($k+1\leq i \leq r)$. Set $\eta_i= \psi_{i}\downarrow_{Z(G)/G{}'}$ for $k+1\leq i \leq r$. One can view $\eta_i$ as a faithful linear character of cyclic group $\frac{Z(G)/G'}{\widehat{K_{i}}}.$
Observe that $\ker (\eta_{i}) = \ker \left( \psi_{i}\downarrow_{{Z(G)}/{G'}} \right) =\big( \ker \psi_{i} \cap Z(G)\big)/G{}'$. 
 On the other hand, let $H_i = \langle a_i \rangle$ for all $i$ ($1\leq i \leq r$) and let $\widehat{H_i} = H_1 \times \cdots \times H_{i-1} \times H_{i+1} \times \cdots \times H_{r}$. Now, for each $i$ ($1\leq i \leq k $),  let $\chi_i \in \nl(G)$ such that  $\chi_i\downarrow_{Z(G)} = \chi_i(1) \lambda_i,$ where $\lambda_{i} \in \lin(Z(G))$ and $\ker (\chi_i)=\ker (\lambda_i)=\widehat{H_i}$. One can view $\lambda_i$ as a faithful linear character of cyclic group $Z(G)/\widehat{H_i}$. Now,
$$ \bigcap_{i=k+1}^{r} \ker (\psi_{i}\downarrow_{Z(G)/G{}'})= \bigcap_{i=k+1}^{r} \widehat{K_i}= K_1\times \cdots \times K_k=\langle a_1G',\ldots,  a_kG' \rangle = H_1\times \cdots \times H_k,$$ 
and 
$$ \bigcap_{j=1}^{k} \ker (\chi_{j})= \bigcap_{j=1}^{k} \widehat{H_j}= H_{k+1}\times \cdots \times H_r \leq Z(G).$$ 
This shows that  
$\left( \bigcap_{i=1}^{k} \ker (\chi_{i}) \right) \bigcap \left( \bigcap_{i=k+1}^{r} \ker (\psi_{i}) \right)    = 1.$
Therefore,
\begin{equation} \label{eq:3}
X = \{ \chi_{i}~ |~ 1\leq i\leq k \} \cup \{ \psi_{i}~ |~ k+1\leq i\leq r  \} 
\end{equation}
is a set of $r$ irreducible characters which satisfies equation \eqref{eq:X_G}. Now we will prove that set $X$ is a minimal faithful quasi-permutation representation of $G$.\\
\noindent	For the non-linear irreducible characters $\chi_{i},~ 1\leq i \leq k,$ it is easy to observe the following.
\begin{enumerate}
	\item For each $i$, $\Gamma (\chi_i)=\Gamma(\lambda_i)$ and $|\Gamma(\lambda_i)|= \phi(p^{l_{i}}) = p^{l_{i}} - p^{l_{i}-1}$, where $\phi$ is Euler phi function.
	\item For each $i$, $m(\chi_i)=\chi_i(1)m(\lambda_i)$ and $\chi_i(1)=|G/Z(G)|^{1/2}$.
	\item For each $i$, $m(\lambda_i)=\big|\min \big\{\sum_{\sigma \in \Gamma(\lambda_i)}\lambda_i^{\sigma}(g) ~|~ g\in G \big\}\big|=|-p^{l_{i}-1}|= p^{l_{i}-1}$. 
\end{enumerate}

\noindent In the case of the linear characters $\psi_{i}$ ($k+1\leq i \leq r)$, we first note that $\frac{G/G{}'}{\ker (\psi_{i})}$ is cyclic. For each $i$, let
\begin{equation} \label{eq:11}
\frac{G/G'}{\ker (\psi_{i})} \cong C_{p^{m_{i}}}, ~k+1\leq i \leq r.
\end{equation}
Without loss of generality, we can assume that $p^{m_{k+1}}\leq p^{m_{k+2}}\leq \cdots \leq p^{m_{r}}$. Again, one can view $\psi_i$ as a faithful linear character of cyclic group $G/\ker (\psi_i)$, for $k+1\leq i \leq r$.
Now $|\Gamma(\psi_{i})|\psi_{i}(1) =  \phi(p^{m_i})\cdot 1 =  p^{m_i} - p^{m_i-1}$, for each $i$ ($k+1\leq i \leq r)$. Also, it is easy to observe that $m(\psi_{i}) = \big|\min \big\{\sum_{\sigma \in \Gamma(\psi_i)}\psi_i^{\sigma}(g) ~|~ g\in G \big\}\big| = |-p^{m_{i}-1}| = p^{m_{i}-1}.$
Now consider the character $$\xi=\sum_{i=1}^r \theta_i,$$ where $\theta_i=\sum_{\sigma\in \Gamma(\chi_i)}\chi_i^{\sigma}= \chi_i(1)\bigg(\sum_{\sigma \in \Gamma(\lambda_i)}\lambda_i^{\sigma}\bigg)=|G/Z(G)|^{1/2}\bigg(\sum_{\sigma \in \Gamma(\lambda_i)}\lambda_i^{\sigma}\bigg),$ for $1\leq i \leq k,$ and $\theta_i=\sum_{\sigma\in \Gamma(\psi_i)}\psi_i^{\sigma},$ for $k+1\leq i \leq r.$
Then by the above observations, we have $$m(\xi)= |G/Z(G)|^{1/2} \sum_{i=1}^{k} p^{l_{i}-1} + \sum_{k+1}^{r} p^{m_{i}-1}$$ and $$\xi(1)= |G/Z(G)|^{1/2} \sum_{i=1}^{k} p^{l_{i}-1}(p-1) + \sum_{k+1}^{r} p^{m_{i}-1}(p-1).$$ Therefore, 
\begin{equation*} \label{eq:9}
c(G) \leq \xi(1)+m(\xi)=|G/Z(G)|^{1/2} \sum_{i=1}^{k} p^{l_{i}} + \sum_{k+1}^{r} p^{m_{i}} = |G/Z(G)|^{1/2} \sum\limits_{i=1}^{k} |H_{i}| + \sum\limits_{i=k+1}^{r} p^{m_{i}}.
\end{equation*}	
Thus we have got a collection $X$, from equation \eqref{eq:3}, which satisfies equation \eqref{eq:X_G} and we get $c(G)\leq |G/Z(G)|^{1/2} \sum\limits_{i=1}^{k} |H_{i}| + \sum\limits_{i=k+1}^{r} p^{m_{i}}$.\\

\noindent {\bf Step II: $\mathbf{c(G) = \xi(1) + m(\xi)}$.}\\
\noindent In the set $X$ given by \eqref{eq:3}, we have exactly $k$ many non-linear irreducible characters and $r-k$ many linear characters and also no proper subset of $X$ satisfies \eqref{eq:X_G}. Thus the set $X$ satisfies Lemma \ref{L6}. In addition to this, set $X$ gives a set $Y= \{ \lambda_{i}~ |~ 1\leq i\leq k \} \cup \{ \eta_{i}~ |~ k+1\leq i\leq r  \}$, where $\lambda_i$ for 
$1\leq i \leq k$ and $\eta_i$ for  $k+1\leq i \leq r$ are defined above. It is easy to see that the set $Y$, which contains $r$ irreducible characters of $Z(G)$ gives $c(Z(G))=\widehat{\xi}(1) + m(\widehat{\xi})$, where 
$\widehat{\xi}=\sum_{i=1}^r \widehat{\theta_i},$ where $\widehat{\theta_i}=\sum_{\sigma\in \Gamma(\lambda_i)}\lambda_i^{\sigma}$ for $1\leq i \leq k,$ and $\widehat{\theta_i}=\sum_{\sigma\in \Gamma(\eta_i)}\eta_i^{\sigma},$ for $k+1\leq i \leq r.$ 
This shows that the set $X$ gives $c(G)$, i.e.,
$$c(G)= |G/Z(G)|^{1/2} \sum\limits_{i=1}^{k} |H_{i}| + \sum\limits_{i=k+1}^{r} p^{m_{i}}.$$
Also, from Theorem \ref{thm:pgroup}, $c(G)=\mu(G) = q(G)$ for $p\neq 2,$  and thus for VZ $p$-group, for odd prime $p$, $c(G)=\mu(G)=q(G)=|G/Z(G)|^{1/2} \sum\limits_{i=1}^{k} |H_{i}| + \sum\limits_{i=k+1}^{r} p^{m_{i}}$.
This proves the result.\\

\noindent {\bf Proof of Corollary \ref{C2}:}
	From Lemma \ref{L6}, we know that a minimal faithful quasi-permutation representation $X_G$ has exactly $r-k$ linear characters. As in our case, $r=k$, $X_G$ does not contain any linear characters.\\
	 Thus, we consider all the $r$ characters to be non-linear irreducible. Let $H_{i}, ~ 1\leq i \leq r,$ denote the $r$ cyclic components of $Z(G)$ such that $|H_{i}|=p^{l_{i}},$ and $\widehat{H_{i}} = H_{1} \times \cdots \times H_{i-1} \times H_{i+1} \times \cdots \times H_{r}.$  Let $\chi_{i} \in \Irr(G|Z(G)),$ for all $i,~1\leq i \leq r,$ be such that ${\chi_{i}}{}\downarrow_{Z(G)} = \chi_{i}(1)\lambda_{i},$ where $\lambda_{i}\in \lin(Z(G))$ is a faithful linear character of $Z(G)/\widehat{H_{i}}.$ 
	  Then 
	  $ |\Gamma( \chi_{i} )| \chi_{i}(1) = \phi(|H_{i}|) |G/Z(G)|^{1/2} = (p^{l_{i}}-p^{l_{i}-1})|G/Z(G)|^{1/2} $ and $m(\chi_i) = |G/Z(G)|^{1/2} \left| - p^{l_{i}-1} \right| = |G/Z(G)|^{1/2}p^{l_{i}-1},$ for each $i$.
	  Now consider the character $\xi=\sum_{i=1}^r \theta_i$, where $\theta_i=\sum_{\sigma\in \Gamma(\chi_i)}\chi_i^{\sigma}= \chi_i(1)\bigg(\sum_{\sigma \in \Gamma(\lambda_i)}\lambda_i^{\sigma}\bigg)=|G/Z(G)|^{1/2}\bigg(\sum_{\sigma \in \Gamma(\lambda_i)}\lambda_i^{\sigma}\bigg),$ for $1\leq i \leq r$.
	  Then by the above observations, we have $$m(\xi)= |G/Z(G)|^{1/2} \sum_{i=1}^{r} p^{l_{i}-1} $$ and $$\xi(1)= |G/Z(G)|^{1/2} \sum_{i=1}^{r} p^{l_{i}-1}(p-1).$$
	   Then 
	  $\xi(1)+m(\xi) = 
	   |G/Z(G)|^{1/2} \sum_{i=1}^{r} |H_{i}| = |G/Z(G)|^{1/2} c(Z(G))$.\\
	 Note that the set $X = \{ \chi_{i} \}_{i=1}^{r}$ gives a set $Y = \{ \lambda_i \}_{i=1}^{r}$, which is a minimal faithful quasi-permutation representation of $Z(G)$. 
	  This shows that
	 \begin{align*}
	 c(G) = \xi(1)+m(\xi) = |G/Z(G)|^{1/2} c(Z(G)). 
	 \end{align*}

\begin{remark}\label{remark:CorC2}
\noindent \begin{enumerate}
\item The converse of the Corollary \ref{C2} is not true in general, i.e., there exists VZ $p$-group for which $c(G) = |G/Z(G)|^{1/2} c(Z(G))$ but $d(Z(G)) > d(G')$. For example, consider the group  
	 	\[ G = \phi_{2}(22) = \langle \alpha, \alpha_{1}, \alpha_{2}~|~ [\alpha_{1}, \alpha] = \alpha^{p} = \alpha_{2}, \alpha_1^{p^2}=\alpha_{2}^{p} = 1 \rangle, \]
which is a VZ $p$-group of order $p^4$ ($p$ an odd prime)	 	(see \cite[Section 4.4]{RJ}). Here $Z(G) = \langle \alpha^p, \alpha_{1}^p \rangle \cong C_{p} \times C_p$, $G' = \langle \alpha^p \rangle \cong C_p$ and $|G/Z(G)|^{1/2} = p$. Thus, $d(Z(G)) > d(G')$. However, $c(G) = 2p^{2} = |G/Z(G)|^{1/2} c(Z(G))$ (we have provided the details of the computation of $c(\phi_{2}(22))$ in subsection \ref{ss:pgroups}). 
 \item 	There exists VZ $p$-group such that $d(Z(G))> d(G{}')$ but $c(G) < |G/Z(G)|^{1/2} c(Z(G))$. 
	For example, consider the group
	\[ G = \phi_{2}(321)e = \phi_{2}(21) \times C_{p^3} = \langle \alpha, \alpha_{1}, \alpha_{2}~|~ [\alpha_{1}, \alpha] = \alpha_{2}, \alpha^{p}=\alpha_{2}, \alpha_{1}^p = \alpha_{2}^p = 1 \rangle \times C_{p^3},  \]
	which is a VZ $p$-group of order $p^6$ ($p$ an odd prime) (see \cite[Section 4.6]{RJ}). 
Now, let $G_1 = \phi_{2}(21)$. Then $Z(G_1) = G{}'_1 = \langle \alpha^p  \rangle \cong C_p$. As $G= G_1 \times C_{p^3}$, $c(Z(G)) =c(Z(G_1))+c(C_{p^3})=p+ p^3$. Now, from Corollary \ref{C2}, $c(G_1) = |G_1/Z(G_1)|^{1/2} c(Z(G_1)) = p^2$. Then from Theorem \ref{thm:nilpotent} and \ref{thm:pgroup}, we get $c(G) = p^2 + p^3 < |G/Z(G)|^{1/2} c(Z(G))=p (p+p^3)=p^2+p^4$. 	
	 	\end{enumerate}
	 \end{remark}
	
%	\begin{corollary} \label{C3}
%	Let $G$ be a $p$-group of order $p^{n}$ and let $|Z(G)|=p^{n-2}.$ Let $d(Z(G)) = r, d(G')= k,$ and $Z(G)$ be given by eq. \eqref{eq:1}. Then 
%	\begin{equation} \label{eq:18}
%	c(G) =  \sum\limits_{i=1}^{k}p^{l_{i}+1} + \sum\limits_{i=k+1}^{r} p^{m_{i}},
%	\end{equation}
%	where $C_{p^{m_{i}}},$ for each $i, ~ k+1\leq i \leq r,$ is as described in eq. \eqref{eq:11}. Moreover, when $p$ is an odd prime, then we get the expression in eq. \eqref{eq:18} for $\mu(G)$ as well.
%	\end{corollary}
%\noindent	\emph{Proof.} Clearly, $\cd(G) = \{ 1, p \}.$ Then Corollary 2.29 of \cite{I} implies that $G$ is a $VZ$ $p$-group. Also, from Theorem 3.2 of \cite{BG}, $c(G)=\mu(G)$ for $p\neq 2,$  and the result follows.\\
	
	\noindent {\bf Proof of Corollary \ref{C3}.}
	Let $G$ be a VZ-group of odd order. As $G$ is a nilpotent group, it is isomorphic to the direct product of its Sylow subgroups. Since $G$ has exactly two character degrees, all the Sylow subgroups of $G$, except one, are abelian. Let $G \cong P_{1} \times P_{2} \times \cdots \times P_{n},$ where $P_{i}$ is a sylow $p_{i}-$subgroup of $G,$ for each $i,$ $ 1\leq i\leq n,$ where $P_{1}$ is the only non-abelian VZ $p_{1}$-group and the rest are abelian Sylow subgroups. Then 
	\begin{align*}
	\mu(G) &= \mu(P_{1})+ \mu(P_{2}) + \cdots + \mu(P_{n}) \quad (\text{by Theorem \ref{thm:nilpotent}}) \\
	&= c(P_{1})+ c(P_{2})+ \cdots c(P_{n}) \quad (\text{by Theorem \ref{thm:pgroup}})  \\
	& = c(G) \quad (\text{see \cite{MG}}).
	\end{align*}
	 Thus, we finally get the values of $\mu(G)$, $q(G)$ and $c(G)$ by using Theorem \ref{T1}. 
	
\section{Isoclinism and VZ $p$-group of order $p^6$}\label{sec:isoclexam}
In this section, we use the notation of the paper \cite{RJ}. We will use not only the results of the
previous sections but, more crucially, also use the classification of p-groups provided by James (\cite{RJ}).
The aim of this section is to calculate $c(G)$ for all the VZ $p$-groups of order $p^6$ ($p$ an odd prime).   
We begin with the definition. 	
\begin{definition} Two  finite groups $G$ and $H$ are said to be isoclinic if
there exist isomorphisms $\theta : G/Z(G)\longrightarrow H/Z(H)$ and $\phi: G{}'\longrightarrow H{}'$
such that the following diagram is
commutative:
  \begin{equation*}
    \begin{tikzcd}
      G/Z(G)\times G/Z(G) \arrow{d}{\theta \times \theta} \arrow{r}{a_G}
      & G{}' \arrow{d}{\phi} \\
      H/Z(H)\times H/Z(H) \arrow{r}{a_H}
      & H{}',
    \end{tikzcd}
  \end{equation*}
  where $a_G(g_1Z(G), g_2Z(G)) = [g_1,g_2]$, for $g_1,g_2\in G$, and $a_H(h_1Z(H), h_2Z(H)) = [h_1,h_2]$ for $h_1,h_2\in H$.
\end{definition}
\noindent The resulting pair $(\theta, \phi)$ is called an \emph{isoclinism} of $G$ onto $H$.
Notice that isoclinism is an equivalence relation on groups, weaker than isomorphism and it was first introduced by P. Hall \cite{Hall} for a classification of $p$-groups. It is well-known that two isoclinic nilpotent groups have the same nilpotency class. Let the number of the irreducible characters of $G$ of degree $k$ be denoted by $r_k(G)$. Then we have the following theorem.

\begin{theorem}\label{thm:isoclinicdegree}\textnormal{\cite[Theorem 2.2]{Bioch}} Let $G$ and $H$ be isoclinic groups. Then
\begin{enumerate}
\item $|H|r_k(G) = |G|r_k(H).$
\item The matrices of the irreducible complex representations
of $G$ and $H$ only differ by scalar factors.
\end{enumerate}
\end{theorem}	
	
\begin{corollary} \label{extra1}
Let $G$ and $H$ be two isoclinic VZ $p$-groups of same order such that $d(Z(G)) = d(G{}')$. Then $c(G) = c(H)$ if and only if $Z(G)$ is isomorphic to $Z(H)$.
\end{corollary}
\begin{proof} Result follows from Corollary \ref{C2}.
\end{proof} 
	
\noindent {\bf Note.} If we relax the condition $d(Z(G)) = d(G{}')$ in Corollary \ref{extra1}, then Corollary \ref{extra1} is no longer true. To produce counter examples, we have taken groups of order $p^6$, mentioned in the classification of $p$-groups, which is based on the concept of isoclinism (see \cite{RJ}). Groups of order $p^{6}$ (odd prime $p$) have been divided into 43 isoclinic families denoted by $\Phi_k$, $1\leq k \leq 43$ (\cite[Section 4.6]{RJ}). The groups are given by polycyclic presentation, in which all the relations of the form $[x,y] =x^{-1}y^{-1}xy = 1$ between the generators have been omitted from the list. Since in the calculation of $c(G)$, we heavily use the presentations of the groups of order $p^6$ (p odd) from
the paper of James (\cite{RJ}), the reader
is advised to keep this paper handy. VZ $p$-groups of order $p^6$ belong to isoclinic families $\Phi_{k}$, for $k=2, 5, 15$.

\begin{example}\label{example1}
{\bf Isoclinic VZ $p$-groups $G$ and $H$ such that $d(Z(G)) \neq d(G{}')$ and $Z(G)\cong Z(H)$ but $c(G) \neq c(H)$}.\\
\textnormal{
	\noindent  Consider two groups, $G$ and $H$, of order $p^{6}$ ($p$ an odd prime) given by
\[ G = \phi_{2}(3111)a = \phi_{2}(31) \times C_{p} \times C_{p} = \langle \alpha, \alpha_{1}, \alpha_{2}~|~ [\alpha_{1}, \alpha] = \alpha^{p^2} = \alpha_{2}, \alpha_{1}^p = \alpha_{2}^p = 1 \rangle \times C_{p} \times C_{p}  \]
and
\[ H = \phi_{2}(222)a = \phi_{2}(22) \times C_{p^2} = \langle \alpha, \alpha_{1}, \alpha_{2}~|~ [\alpha_{1}, \alpha] = \alpha^{p} = \alpha_{2}, \alpha_1^{p^2}=\alpha_{2}^{p} = 1 \rangle \times C_{p^2}, \]
which belong to the isoclinic family $\Phi_{2}$ (\cite{RJ}). From Theorem \ref{thm:nilpotent} and \ref{thm:pgroup}, $c(G) = c(\phi_{2}(31)) + c(C_p) + c(C_p)$, and $c(H) = c(\phi_{2}(22)) + c(C_{p^2})$, and thus we need to calculate $c(\phi_{2}(31))$ and $c(\phi_{2}(22))$.
Let $\phi_{2}(31) = G_1$. Then $Z(G_1) = \langle \alpha^p \rangle \cong C_{p^2}$ and thus from Theorem \ref{thm:cycliccenter}, $c(G_1) = p^3$, which finally gives $c(G)= p^3+2p$.
Now, let $\phi_{2}(22) = H_1$. Then $Z(H_1) = \langle \alpha^p, \alpha_{1}^p \rangle \cong C_{p} \times C_p$, ${H}_1' = \langle \alpha^p \rangle$ and $Z(H_1)/H_1' = \langle \alpha_1^p H_1' \rangle$. We have computed $c(H_1)$ in subsection \ref{ss:pgroups}, and we get $c(H_1) = 2p^2$, and finally $c(H) = 2p^2 +p^2 = 3p^2$.}
\end{example}
%Here $d(Z(G)) = 2$ and $d(G') = 1$, and so we have to use two irreducible characters of $G$ to obtain a minimal faithful quasi-permutation representation. \\
%Then from Theorem \ref{T1}, we consider $X = \{ \chi_{1}, \psi_{2} \}$, where $\chi_1 \in nl(H_1)$ and $\psi_{2} \in lin(H_1).$     
%Here $\chi_{1}\downarrow_{Z} = \chi_{1}(1) \lambda_{1},$ where $\lambda_{1} \in lin(Z(H_1))$ is given by $\lambda_{1} = \lambda'_{\langle \alpha^p \rangle} \cdot 1_{\langle \alpha_{1}^p \rangle} $, where $\lambda'$ is a faithful linear character of $\langle \alpha^p  \rangle$. On the other hand $\psi_{2}\downarrow_{Z(H_1)/H_1{}'} = \psi'_{\langle \alpha_1^p H_1' \rangle} $, where $\psi'$ is a faithful linear character of $\langle \alpha_1^p  \rangle$.
%Then $\ker \chi_{1} = \langle \alpha_{1}^p \rangle,$ and  $\ker \left( \psi_{2}\downarrow_{Z(H_1)/H{}'_{1}} \right) = \langle \alpha^p \rangle$. Let 
%$\xi=\sum_{i=1}^2 \theta_i,$ where $\theta_1=\sum_{\sigma\in \Gamma(\chi_1)}\chi_1^{\sigma}= p\bigg(\sum_{\sigma \in \Gamma(\lambda_1)}\lambda_1^{\sigma}\bigg),$ and $\theta_2=\sum_{\sigma\in \Gamma(\psi_2)}\psi_2^{\sigma}$.
%Then, $m(\xi)= p + p $ and $\xi(1)= p(p-1) + p(p-1) $. Therefore, 
%$c(H_1) = \xi(1) + m(\xi) = 2 p^2$ and thus,
% $c(H) = 2p^2 +p^2 = 3p^2$.

\begin{example}\label{example2}
{\bf Isoclinic VZ $p$-groups $G$ and $H$ such that $d(Z(G)) \neq d(G{}')$ and $Z(G)\ncong Z(H)$ with $c(G) \neq c(H)$}.\\
\textnormal{
	\noindent  Let us consider two groups  $G = \phi_{2}(3111)a = \phi_{2}(31) \times C_p \times C_p$, and $H = \phi_{2}(321)b = \phi_{2}(31) \times C_{p^2}$ of order $p^{6}$ ($p$ an odd prime) which belong to the isoclinic family $\Phi_{2}$ (\cite{RJ}). From Example \ref{example1}, we get $c(\phi_{2}(31))=p^3$. Here $Z(G) = \langle \alpha^p \rangle \times C_p \times C_p \cong C_{p^2} \times C_p \times C_p$, $Z(H) = \langle \alpha^p \rangle \times C_{p^2} \cong C_{p^2} \times C_{p^2}$ and $d(Z(G)) \neq d(G')$. Again by Theorem \ref{thm:nilpotent} and \ref{thm:pgroup}, we get $c(G) = p^3 + 2p$ and $c(H) = p^3 + p^2$.}
		
\end{example}

\begin{example}\label{example3}
	
{\bf Isoclinic VZ $p$-groups $G$ and $H$, such that $d(Z(G))\neq d(G{}')$ and $Z(G)\cong Z(H)$ with $c(G)= c(H)$.}\\
\textnormal{
\noindent Consider two groups of order $p^5$ ($p$ an odd prime) given by 
\[ G = \phi_{2}(32)a_1 = \langle \alpha, \alpha_{1}, \alpha_{2} ~|~ [\alpha_{1}, \alpha] = \alpha^{p^2} = \alpha_{2}, \alpha_{1}^{p^2}= \alpha_{2}^p = 1 \rangle  \]
and 
\[ H = \phi_{2}(32)a_2 = \langle \alpha, \alpha_{1}, \alpha_{2}~|~ [\alpha_{1}, \alpha] = \alpha_1^{p}=\alpha_{2}, \alpha^{p^3}=\alpha_{2}^{p} = 1 \rangle,  \]
which belong to the isoclinic family $\Phi_{2}$ (\cite{RJ}).
Here $Z(G) = Z(H) = \langle \alpha^p, \alpha_{1}^p \rangle \cong C_{p^2} \times C_p$, $G{}' = \langle \alpha^{p^2} \rangle \cong C_p$, $H{}' = \langle \alpha_1^{p} \rangle \cong C_p$, $Z(G)/G' = \langle \alpha^p G{}', \alpha_1^p G{}' \rangle \cong C_p \times C_p$ and $Z(H)/H' = \langle \alpha^p H{}' \rangle \cong C_{p^2}$.
Further, $d(Z(G)) = d(Z(H)) = 2$ and $d(G') = d(H') = 1$. We first compute $c(G)$ and then $c(H)$.
In the view of the proof of Theorem \ref{T1}, we consider a set $X = \{ \chi_{1}, \psi_{2} \}$, where $\chi_1 \in \nl(G)$ and $\psi_{2} \in \lin(G/G{}').$     
Here $\chi_{1}\downarrow_{Z(G)} = \chi_{1}(1) \lambda_{1},$ where $\lambda_{1} \in \lin(Z(G))$ is given by $\lambda_{1} = \lambda'_{\langle \alpha^p \rangle} \cdot 1_{\langle \alpha_{1}^p \rangle} $, where $\lambda'_{\langle \alpha^p \rangle}$ is a faithful linear character of $\langle \alpha^p  \rangle$. On the other hand, $\psi_{2} = 1_{\langle \alpha G' \rangle} \cdot \psi_{\langle \alpha_1 G' \rangle} $, where $\psi_{\langle \alpha_1 G' \rangle}$ is a faithful linear character of $\langle \alpha_1 G' \rangle$, so that $\psi_{2}\downarrow_{Z(G)/G{}'} = 1_{\langle \alpha^p G' \rangle} \cdot \psi'_{\langle \alpha_1^p G' \rangle} $, where $\psi{}'_{\langle \alpha_1^p G' \rangle}$ is a faithful linear character of $\langle \alpha_1^p G' \rangle$. 
Then $\ker \chi_{1} = \langle \alpha_{1}^p \rangle,$ and  $\ker \left( \psi_{2}\downarrow_{Z(G)/G{}'} \right) = \langle \alpha^p \rangle$. Consider 
$\xi=\sum_{i=1}^2 \theta_i,$ where $\theta_1=\sum_{\sigma\in \Gamma(\chi_1)}\chi_1^{\sigma}= p\bigg(\sum_{\sigma \in \Gamma(\lambda_1)}\lambda_1^{\sigma}\bigg),$ and $\theta_2=\sum_{\sigma\in \Gamma(\psi_2)}\psi_2^{\sigma}$.
Then, $m(\xi)= p^2 + p $ and $\xi(1)= p^2(p-1) + p(p-1) $. Therefore, 
$c(G) = \xi(1) + m(\xi) = p^3 + p^2$.
In the similar way, we get $c(H) = p^3 + p^2$.}
\end{example}
%\noindent From above example, it is clear that when $d(Z(G)) \neq d(G')$ and $Z(G) \cong Z(H)$, we may still get $c(G) \neq c(H)$. 

\begin{example}\label{example4}
{\bf Isoclinic VZ $p$-groups $G$ and $H$ such that $d(Z(G)) \neq d(G{}')$ and $Z(G)\ncong Z(H)$ but $c(G) = c(H)$}.\\
\textnormal{
	\noindent  Consider two groups, $G$ and $H$, of order $p^{6}$ ($p$ an odd prime) given by
		\[ G = \phi_{2}(321)e = \phi_{2}(21) \times C_{p^3} = \langle \alpha, \alpha_{1}, \alpha_{2}~|~ [\alpha_{1}, \alpha] = \alpha_{2}, \alpha^{p}=\alpha_{2}, \alpha_{1}^p = \alpha_{2}^p = 1 \rangle \times C_{p^3},  \]
	 and $H = \phi_{2}(321)b = \phi_{2}(31) \times C_{p^2}$, which belong to the isoclinic family $\Phi_{2}$ (\cite{RJ}). Here $Z(G) = \langle \alpha^p \rangle \times C_{p^3} \cong C_{p} \times C_{p^3}$, $Z(H) = \langle \alpha^p \rangle \times C_{p^2} \cong C_{p^2} \times C_{p^2}$ and $d(Z(G)) \neq d(G{}')$. From Example \ref{example2}, $c(H) = p^3 + p^2$.
	 On the other hand, from Remark \ref{remark:CorC2}, we get $c(G) = p^3 + p^2$ and hence $c(G) = c(H)$.}
\end{example}
	
%\[  G = \phi_{5}(21^4)b = \langle \alpha_{1}, \alpha_{2}, \alpha_3, \alpha_4, \gamma ~|~ [\alpha_{1}, \alpha_2]= [\alpha_{3}, \alpha_4] = \gamma^p, \alpha_{1}^p = \alpha_{2}^p = \alpha_{3}^p =\alpha_4^p= \gamma^{p^2} = 1 \rangle.  \] 
%	Here $Z(G) =\langle \gamma \rangle \cong C_{p^2} $ and $G' = \langle \gamma^p \rangle$. Then, it is easy to verify that, for all $g\in G\setminus Z(G)$, the conjugacy class of $g$ in $G$ is $g \langle \gamma^p \rangle = gG'$. Hence, $G$ is a VZ $p$-group and thus $\cd(G) = \{ 1, |G/Z(G)|^{1/2} \} = \{ 1, p^2 \}$.	

%\noindent \textcolor{blue}{Now we provide complete list of $c(G)$, where $G$ is a $p$-group of order $p^6$.}	
\noindent In the next subsection, we give a complete list of $c(G)$ for a VZ $p$-groups of order $p^6$.
\subsection{VZ $p$-groups of order $p^6$ ($p$ an odd prime)} \label{ss:pgroups}	

\begin{corollary} \label{extra2}
Let $G$ be a group of order $p^6$ belonging to $\Phi_{15}$ (odd prime $p$). Then $c(G) = 2p^3.$
	\end{corollary}
\begin{proof}
Consider the group 
	\begin{align*}
	G = \phi_{15}(1^6) &=\langle \alpha_{1}, \alpha_{2}, \alpha_3, \alpha_4, \beta_1, \beta_2 ~|~ [\alpha_{1}, \alpha_2]= \beta_1, [\alpha_{1}, \alpha_{3}] = \beta_2,\\  &[\alpha_{3}, \alpha_4] = \beta_1, [\alpha_{2}, \alpha_4] = \beta_2^q, \alpha_{1}^p = \alpha_{2}^p = \alpha_{3}^p =\alpha_4^p= \beta_1^p = \beta_2^p = 1 \rangle,
	\end{align*}
	where $q$ denotes the smallest positive integer which is a primitive root$\Mod p$ (\cite{RJ}). The group $G$ is a VZ $p$-group.  Here $Z(G) = G{}' =\langle \beta_1, \beta_2  \rangle \cong C_{p} \times C_p $ and $ \cd(G) = \{ 1, p^2 \}$.  Now by Theorem \ref{thm:isoclinicdegree}, for all the groups $H\in \Phi_{15}$, $\cd(H)=\{1,p^2\}$ and hence all the groups belonging to family $\Phi_{15}$ are VZ groups. Observe that, for each $G \in \Phi_{15}$, $Z(G) = G' \cong C_p \times C_p$. Thus from Corollary \ref{C2} and Corollary \ref{extra1}, we have $c(G) = 2p^3$ for every $G\in \Phi_{15}$.	
\end{proof}	
	
	\noindent In Table \ref{t:1} and \ref{t:2}, we provide a complete list of $c(G)$, where $G$ is a $p$-group of order $p^6$ lying in $\Phi_{2}$ and $\Phi_5$. Calculation of $c(G)$ is based upon the construction of a set $X_G\subset \Irr(G)$ (see in the proof of Theorem \ref{T1}). Here we illustrate how to read Table \ref{t:1} and \ref{t:2}.\\\\
	Table \ref{t:1} contains the groups belonging to family $\Phi_{2}$ and Table \ref{t:2} contains the groups belonging to family $\Phi_5$. For $G \in \Phi_{2}$, $|G/Z(G)| = p^2$ and thus $\cd(G) = \{ 1, p \}$. For $G \in \Phi_5$, $|G/Z(G)| = p^4$ and hence $\cd(G) = \{ 1, p^2 \}$.
%	The computation of $c(G)$ follows the proof of Theorem \ref{T1} and we look for a minimal faithful quasi-permutation set $X \subset \Irr(G)$. 
There are two types of groups: 
\begin{enumerate}
\item Type 1: groups which are not a direct product. \item Type 2: groups which are expressed as a direct product of two or more groups. 
\end{enumerate} 
The column with the heading ``Group $G$" contains the name of the group of order $p^6$. 
For groups of Type 2, the column with the heading ``Group $N$" contains the description of  a non-abelian component $N$ of $G$. For groups of Type 2, computation has been  given for $N$ and finally $c(G)$ is computed from Theorems \ref{thm:nilpotent} and \ref{thm:pgroup}. In the column with the heading ``Set $X$", we have mentioned the members of the set $X_G$ for the groups of Type 1, and the members of the set $X_N$ for the groups of Type 2. Note that set $X_G$ is in the form $\{ \chi_{1}, \ldots, \chi_k,\psi_{k+1}, \ldots, \psi_r \}$, where $\chi_{i} \in \nl(G)$ (for $1\leq i \leq k$), and $\psi_j \in \lin(G/G{}')$ (for $k+1 \leq j \leq r$) with  $d(Z(G))=r$ and $d(G{}')=k$. Here $\chi_{i}\downarrow_{Z(G)} = \chi_{i}(1) \lambda_i$, where for each $i$, $\lambda_i \in \lin (Z(G))$ such that $G{}' \nleq \ker(\lambda_i)=\ker(\chi_i)$. But in Table \ref{t:1} and \ref{t:2}, in place of $\chi_i$ for $(1\leq i \leq k)$, we have just mentioned $\lambda_i$. The notation $\lambda_{H}$ stands for a faithful character of a subgroup $H$, and $1_H$ stands for the trivial character of $H$. For better understanding on the working of table, we present an example. Let
	\[ G = \phi_{2}(222)a = \phi_{2}(22) \times C_{p^2} = \langle \alpha, \alpha_{1}, \alpha_2~|~ [\alpha_{1}, \alpha] = \alpha^{p}=\alpha_2, \alpha_1^{p^2}=\alpha_{2}^{p} = 1 \rangle \times C_{p^2} .\]
	We take $N=\phi_{2}(22)$, which is a $p$-group of order $p^4$ with nilpotency class 2. Then $Z(N) = \langle \alpha^p, \alpha_{1}^p \rangle \cong C_{p} \times C_p$, $N{}' = \langle \alpha^p \rangle$, $Z(N)/N{}' = \langle \alpha_1^p N{}' \rangle$, $N/N{}'=\langle \alpha N{}', \alpha_1 N{}'\rangle\cong C_p \times C_{p^2}$ and $(\alpha N{}')^p=(\alpha_1 N{}')^{p^2}=N{}'$.
	Here $d(Z(N)) = 2$ and $d(N{}') = 1$, and so we have to use two irreducible characters of $N$ to obtain a minimal faithful quasi-permutation representation (see Theorem \ref{T1}). Consider a set $X_N = \{ \chi_{1}, \psi_{2} \}\subseteq \Irr(N)$, where $\chi_1 \in \nl(N)$ and $\psi_{2} \in \lin(N/N{}').$     
	Here $\chi_{1}\downarrow_{Z(N)} = \chi_{1}(1) \lambda_{1},$ where $\lambda_{1} \in \lin(Z(N))$ is given by $\lambda_{1} = \lambda'_{\langle \alpha^p \rangle} \cdot 1_{\langle \alpha_{1}^p \rangle} $, where $\lambda{}'_{\langle \alpha^p  \rangle}$ is a faithful linear character of $\langle \alpha^p  \rangle$.  On the other hand, $\psi_{2} = 1_{\langle \alpha N{}'\rangle}\cdot \psi_{\langle \alpha_1 N' \rangle} $, where $\psi_{\langle \alpha_1 N' \rangle}$ is a faithful linear character of $\langle \alpha_1 N' \rangle$. Now, $\psi_{2}\downarrow_{Z(N)/N{}'} = \bigg( 1_{\langle \alpha N{}'\rangle}\cdot \psi_{\langle \alpha_1 N' \rangle}\bigg){}\downarrow_{Z(N)/N{}'}$ is a faithful irreducible character of $Z(N)/N{}'=\langle \alpha_1^p N' \rangle$. 
%	where $\psi'$ is a faithful linear character of $\langle \alpha_1^p N' \rangle$.
	Then $\ker (\chi_{1}) = \langle \alpha_{1}^p \rangle,$ $\ker (\psi_2)= \langle \alpha N{}' \rangle$ and  $\ker \left( \psi_{2}\downarrow_{Z(N)/N{}'} \right) = N{}'=\langle \alpha^p \rangle$. 
Let $\eta_{N{}'}: N \longrightarrow N/N{}'$ be a natural group homomorphism. Then $\widetilde{\psi_2}:=\psi_2 \circ \eta_{N{}'}$ is a linear character of $N$ with $\ker(\widetilde{\psi_2})=\langle \alpha \rangle$. So we can think $\widetilde{\psi_2}$ as a faithful irreducible character of $N/\ker(\widetilde{\psi_2})=\langle \alpha_1\rangle$. Observe that 
\begin{enumerate}
\item $|\Gamma(\widetilde{\psi_2})|= |\Gamma(\psi_2)|=\phi(p^2)=p(p-1)$.
\item $m\bigg(\sum_{\sigma\in \Gamma(\widetilde{\psi_2})}\widetilde{\psi_2}^{\sigma} \bigg)=m\bigg(\sum_{\sigma\in \Gamma(\psi_2)}\psi_2^{\sigma} \bigg)=p$.
\end{enumerate} 
		
Consider 
	$\xi=\sum_{i=1}^2 \theta_i,$ where $\theta_1=\sum_{\sigma\in \Gamma(\chi_1)}\chi_1^{\sigma}= p\bigg(\sum_{\sigma \in \Gamma(\lambda_1)}\lambda_1^{\sigma}\bigg),$ and $\theta_2=\sum_{\sigma\in \Gamma(\psi_2)}\psi_2^{\sigma}$.
	Then, $m(\xi)= p + p=2p $ and $\xi(1)= p(p-1) + p(p-1)=2p(p-1)$. Therefore, 
	$c(N) = \xi(1) + m(\xi) = 2 p^2$. From Theorems \ref{thm:nilpotent} and \ref{thm:pgroup}, $c(G) = c(\phi_{2}(22)) + c(C_{p^2})$,
	and thus, $c(G) = 2p^2 +p^2 = 3p^2$. This group has been mentioned at Sl. no. 10 in Table \ref{t:1}. 
	
	\begin{tiny}
		
		\begin{longtable}[c]{|l|l|l|l|l|}
			\caption{$c(G)$ of VZ $p$-groups of order $p^6$ belonging to $\Phi_{2}$ family\label{t:1}}\\
			
			\hline
%			\multicolumn{5}{| c |}{Begin of Table}\\
%			\hline
			Sl. no. &	Group $G$ & Group $N$ & Set $X$  &  $c(G)$ \\
			& & & &\\
			\hline
			\endfirsthead
			
			\hline
			\multicolumn{5}{|c|}{Continuation of Table \ref{t:1}}\\
			\hline
			Sl. no. &	Group $G$ & Group $N$ & Set $X$ &  $c(G)$ \\ 
			\hline
			\endhead
			
			\hline
			\endfoot
			
			\hline
%			\multicolumn{5}{| c |}{End of Table}\\
%			\hline
			\endlastfoot
			
			1 &	\vtop{\hbox{\strut $\phi_{2}(411)a$ } \hbox{\strut $= \phi_{2}(41) \times C_p $ }} & \vtop{\hbox{\strut $ \phi_{2}(41) $ }} & $\lambda_{1} = \lambda_{\langle \alpha^p \rangle} $ & \vtop{\hbox{\strut $c(G) = c(N) + c(C_p)$ } \hbox{\strut $= p^4 + p$ }}  \\
			\hline
			2 &	\vtop{\hbox{\strut $\phi_{2}(321)a_1$ } \hbox{\strut $= \phi_{2}(32)a_1 \times C_p $ }} & \vtop{\hbox{\strut $ \phi_{2}(32)a_1 $ }} & \vtop{\hbox{\strut $\lambda_{1} = \lambda_{\langle \alpha^p \rangle} \cdot 1_{\langle \alpha_{1}^p \rangle }$, } \hbox{\strut $\psi_2 =  1_{\langle \alpha N' \rangle } \cdot \psi_{\langle \alpha_{1}N' \rangle} $ }}  &  \vtop{\hbox{\strut $c(G) = c(N) + c(C_p)$ } \hbox{\strut $= p^3 + p^2 + p$}}\\
			\hline
			3 &	\vtop{\hbox{\strut $\phi_{2}(321)a_2$ } \hbox{\strut $= \phi_{2}(32)a_2 \times C_p $ }} & \vtop{\hbox{\strut $ \phi_{2}(32)a_2 $ }} & \vtop{\hbox{\strut $\lambda_{1} = 1_{\langle \alpha^p \rangle} \cdot \lambda_{\langle \alpha_{1}^p \rangle }$, } \hbox{\strut $\psi_2 =  \psi_{\langle \alpha N' \rangle } \cdot 1_{\langle \alpha_{1}N' \rangle} $ }}  &  \vtop{\hbox{\strut $c(G) = c(N) + c(C_p)$ } \hbox{\strut $= p^3 + p^2 + p$}}\\
			\hline
			4 &	\vtop{\hbox{\strut $\phi_{2}(321)b$ } \hbox{\strut $= \phi_{2}(31) \times C_{p^2} $ }} & \vtop{\hbox{\strut $ \phi_{2}(31)$ }} & $\lambda_{1} = \lambda_{\langle \alpha^p \rangle} $ & \vtop{\hbox{\strut $c(G) = c(N) + c(C_{p^2})$ } \hbox{\strut $= p^3 + p^2$ }}  \\
			\hline
			5 &	\vtop{\hbox{\strut $\phi_{2}(321)e$ } \hbox{\strut $= \phi_{2}(21) \times C_{p^3} $ }} & \vtop{\hbox{\strut $ \phi_{2}(21)$ }} & $\lambda_{1} = \lambda_{\langle \alpha^p \rangle} $ & \vtop{\hbox{\strut $c(G) = c(N) + c(C_{p^3})$ } \hbox{\strut $=  p^2 + p^3$ }}  \\
			\hline
			6 &	\vtop{\hbox{\strut $\phi_{2}(3111)a$ } \hbox{\strut $= \phi_{2}(31) \times C_{p} \times C_{p} $ }} & \vtop{\hbox{\strut $ \phi_{2}(31)$ }} & $\lambda_{1} = \lambda_{\langle \alpha^p \rangle} $ & \vtop{\hbox{\strut $c(G) = c(N) + 2c(C_{p})$ } \hbox{\strut $= p^3 + 2p$ }}  \\
			\hline
			7 &	\vtop{\hbox{\strut $\phi_{2}(3111)b$ } \hbox{\strut $= \phi_{2}(311)b \times C_{p} $ }} & \vtop{\hbox{\strut $ \phi_{2}(311)b$ }} & $\lambda_{1} = \lambda_{\langle \gamma \rangle}$ & \vtop{\hbox{\strut $c(G) = c(N) + c(C_{p})$ } \hbox{\strut $= p^4 + p$ }}  \\
			\hline
			8 &	\vtop{\hbox{\strut $\phi_{2}(3111)c$ } \hbox{\strut $= \phi_{2}(311)c \times C_p $ }} & \vtop{\hbox{\strut $ \phi_{2}(311)c $ }} & \vtop{\hbox{\strut $\lambda_{1} = 1_{\langle \alpha^p \rangle} \cdot \lambda_{\langle \alpha_{2} \rangle }$, } \hbox{\strut $\psi_2 =  \psi_{\langle \alpha N' \rangle } \cdot 1_{\langle \alpha_{1}N' \rangle} $ }}  &  \vtop{\hbox{\strut $c(G) = c(N) + c(C_p)$ } \hbox{\strut $= p^3 + p^2 + p$}}\\
			\hline
			9 &	\vtop{\hbox{\strut $\phi_{2}(3111)d$ } \hbox{\strut $= \phi_{2}(111) \times C_{p^3} $ }} & \vtop{\hbox{\strut $ \phi_{2}(111)$ }} & $\lambda_{1} = \lambda_{\langle \alpha_2 \rangle} $ & \vtop{\hbox{\strut $c(G) = c(N) + c(C_{p^3})$ } \hbox{\strut $=  p^2 + p^3$ }}  \\
			\hline
			10 &	\vtop{\hbox{\strut $\phi_{2}(222)a$ } \hbox{\strut $= \phi_{2}(22) \times C_{p^2} $ }} & \vtop{\hbox{\strut $ \phi_{2}(22) $ }} & \vtop{\hbox{\strut $\lambda_{1} = \lambda_{\langle \alpha^p \rangle} \cdot 1_{\langle \alpha_{1}^p \rangle }$, } \hbox{\strut $\psi_2 =  1_{\langle \alpha N' \rangle } \cdot \psi_{\langle \alpha_{1}N' \rangle} $ }}  &  \vtop{\hbox{\strut $c(G) = c(N) + c(C_{p^2})$ } \hbox{\strut $= 2p^2 + p^2 = 3p^2$}}\\
			\hline
			11 &	\vtop{\hbox{\strut $\phi_{2}(2211)a$ } \hbox{\strut $= \phi_{2}(22) \times C_{p} \times C_p $ }} & \vtop{\hbox{\strut $ \phi_{2}(22) $ }} & \vtop{\hbox{\strut $\lambda_{1} = \lambda_{\langle \alpha^p \rangle} \cdot 1_{\langle \alpha_{1}^p \rangle }$, } \hbox{\strut $\psi_2 =  1_{\langle \alpha N' \rangle } \cdot \psi_{\langle \alpha_{1}N' \rangle} $ }}  &  \vtop{\hbox{\strut $c(G) = c(N) + 2c(C_p)$ } \hbox{\strut $= 2p^2 + 2p $}}\\
			\hline
			12 & \vtop{\hbox{\strut $\phi_{2}(2211)b$ } \hbox{\strut $= \phi_{2}(21) \times C_{p^2} \times C_p $ }} & \vtop{\hbox{\strut $ \phi_{2}(21)$ }} & See Column 4 at Sl. no. 5 in this table & \vtop{\hbox{\strut $c(G) = c(N) + c(C_{p^2}) + c(C_{p})$ } \hbox{\strut $= p^2 + p^2 +p = 2p^2 + p$ }}  \\
			\hline
			13 &	\vtop{\hbox{\strut $\phi_{2}(2211)c$ } \hbox{\strut $= \phi_{2}(221)c \times C_{p} $ }} & \vtop{\hbox{\strut $ \phi_{2}(221)c $ }} & \vtop{\hbox{\strut $\lambda_{1} = 1_{\langle \alpha^p \rangle} \cdot \lambda_{\langle \gamma \rangle }$, } \hbox{\strut $\psi_2 =  \psi_{\langle \alpha N' \rangle } \cdot 1_{\langle \alpha_{1}N' \rangle} \cdot 1_{\langle \gamma N' \rangle} $ }}  &  \vtop{\hbox{\strut $c(G) = c(N) + c(C_p)$ } \hbox{\strut $= p^3 + p^2 + p $}}\\
			\hline
			14 &	\vtop{\hbox{\strut $\phi_{2}(2211)d$ } \hbox{\strut $= \phi_{2}(221)d \times C_{p} $ }} & \vtop{\hbox{\strut $ \phi_{2}(221)d $ }} & \vtop{\hbox{\strut $\lambda_{1} = 1_{\langle \alpha^p \rangle} \cdot 1_{\langle \alpha_{1}^p \rangle } \cdot \lambda_{\langle \alpha_{2} \rangle } $, } \hbox{\strut $\psi_2 =  \psi_{\langle \alpha N' \rangle } \cdot 1_{\langle \alpha_{1}N' \rangle} \cdot 1_{\langle \alpha_{2} N' \rangle}, $ } \hbox{\strut $\psi_3 =  1_{\langle \alpha N' \rangle } \cdot \psi'_{\langle \alpha_{1}N' \rangle} \cdot 1_{\langle \alpha_{2} N' \rangle} $ }}  &  \vtop{\hbox{\strut $c(G) = c(N) + c(C_p)$ } \hbox{\strut $= 3p^2 + p $}}\\
			\hline
			15 &	\vtop{\hbox{\strut $\phi_{2}(2211)e$ } \hbox{\strut $= \phi_{2}(211)b \times C_{p^2} $ }} & \vtop{\hbox{\strut $ \phi_{2}(211)b$ }} & $\lambda_{1} = \lambda_{\langle \gamma \rangle} $ & \vtop{\hbox{\strut $c(G) = c(N) + c(C_{p^2})$ } \hbox{\strut $= p^3 + p^2$ }}  \\
			\hline
			16 &	\vtop{\hbox{\strut $\phi_{2}(2211)f$ } \hbox{\strut $= \phi_{2}(211)c \times C_{p^2} $ }} & \vtop{\hbox{\strut $ \phi_{2}(211)c $ }} & \vtop{\hbox{\strut $\lambda_{1} = 1_{\langle \alpha^p \rangle} \cdot \lambda_{\langle \alpha_{2} \rangle }$, } \hbox{\strut $\psi_2 =  \psi_{\langle \alpha N' \rangle } \cdot 1_{\langle \alpha_{1}N' \rangle} \cdot 1_{\langle \alpha_2 N' \rangle} $ }}  &  \vtop{\hbox{\strut $c(G) = c(N) + c(C_{p^2})$ } \hbox{\strut $= 2p^2 + p^2 = 3p^2$}}\\
			\hline
			17 &	\vtop{\hbox{\strut $\phi_{2}(21^4)a$ } \hbox{\strut $= \phi_{2}(21) \times C_{p} \times C_p \times C_p $ }} & \vtop{\hbox{\strut $ \phi_{2}(21)$ }} & See Column 4 at Sl. no. 5 in this table & \vtop{\hbox{\strut $c(G) = c(N) + 3c(C_{p})$ } \hbox{\strut $=  p^2 + 3p$ }}  \\
			\hline
			18 &	\vtop{\hbox{\strut $\phi_{2}(21^4)b$ } \hbox{\strut $= \phi_{2}(21) \times C_{p^2} \times C_p $ }} & \vtop{\hbox{\strut $ \phi_{2}(21)$ }} & See Column 4 at Sl. no. 5 in this table & \vtop{\hbox{\strut $c(G) = c(N) + c(C_{p^2}) + c(C_{p})$ } \hbox{\strut $= p^2 + p^2+ p = 2p^2 + p$ }}  \\
			\hline
			19 &	\vtop{\hbox{\strut $\phi_{2}(21^4)c$ } \hbox{\strut $= \phi_{2}(211)c \times C_{p} \times C_p $ }} & \vtop{\hbox{\strut $ \phi_{2}(211)c$ }} & See Column 4 at Sl. no. 16 in this table & \vtop{\hbox{\strut $c(G) = c(N) + 2c(C_{p})$ } \hbox{\strut $= 2 p^2+ 2p = 2p^2 + 2p$ }}  \\
			\hline
			20 &	\vtop{\hbox{\strut $\phi_{2}(21^4)d$ } \hbox{\strut $= \phi_{2}(111) \times C_{p^2} \times C_p $ }} & \vtop{\hbox{\strut $ \phi_{2}(111)$ }} & See Column 4 at Sl. no. 9 in this table & \vtop{\hbox{\strut $c(G) = c(N) + c(C_{p^2}) + c(C_{p})$ } \hbox{\strut $= p^2 + p^2+ p = 2p^2 + p$ }}  \\
			\hline
			21 &	\vtop{\hbox{\strut $\phi_{2}(1^6)$ } \hbox{\strut $= \phi_{2}(111) \times C_{p} \times C_p \times C_p $ }} & \vtop{\hbox{\strut $ \phi_{2}(111)$ }} & See Column 4 at Sl. no. 9 in this table  & \vtop{\hbox{\strut $c(G) = c(N) + 3c(C_{p})$ } \hbox{\strut $=  p^2 + 3p$ }}  \\
			\hline
			22 &	\vtop{\hbox{\strut $\phi_{2}(51)$ }} & - & $\lambda_{1} = \lambda_{\langle \alpha^p \rangle} $ & \vtop{\hbox{\strut $c(G) =  p^5 $ } }  \\
			\hline
			23 &	\vtop{\hbox{\strut $\phi_{2}(42)a_1$ } } & - & \vtop{\hbox{\strut $\lambda_{1} = \lambda_{\langle \alpha^p \rangle} \cdot 1_{\langle \alpha_{1}^p \rangle }$, } \hbox{\strut $\psi_2 =  1_{\langle \alpha G' \rangle } \cdot \psi_{\langle \alpha_{1}G' \rangle} $ }}  &  \vtop{\hbox{\strut $c(G) = p^4+p^2$ } }\\
			\hline
			24 &	\vtop{\hbox{\strut $\phi_{2}(42)a_2$ } } & - & \vtop{\hbox{\strut $\lambda_{1} = 1_{\langle \alpha^p \rangle} \cdot \lambda_{\langle \alpha_{1}^p \rangle }$, } \hbox{\strut $\psi_2 =  \psi_{\langle \alpha G' \rangle } \cdot 1_{\langle \alpha_{1}G' \rangle} $ }}  &  \vtop{\hbox{\strut $c(G) = p^4+p^2$ } }\\
			\hline
			25 &	\vtop{\hbox{\strut $\phi_{2}(411)b$ }} & - & $\lambda_{1} = \lambda_{\langle \gamma \rangle}$ & \vtop{\hbox{\strut $c(G) =  p^5 $ } }  \\
			\hline
			26 &	\vtop{\hbox{\strut $\phi_{2}(411)c$ } } & - & \vtop{\hbox{\strut $\lambda_{1} = 1_{\langle \alpha^p \rangle} \cdot \lambda_{\langle \alpha_{2} \rangle }$, } \hbox{\strut $\psi_2 =  \psi_{\langle \alpha G' \rangle } \cdot 1_{\langle \alpha_{1}G' \rangle} $ }}  &  \vtop{\hbox{\strut $c(G) = p^4+p^2$ } }\\
			\hline
			27 &	\vtop{\hbox{\strut $\phi_{2}(33)$ } } & - & \vtop{\hbox{\strut $\lambda_{1} = \lambda_{\langle \alpha^p \rangle} \cdot 1_{\langle \alpha_{1}^p \rangle }$, } \hbox{\strut $\psi_2 =  1_{\langle \alpha G' \rangle } \cdot \psi_{\langle \alpha_{1}G' \rangle} $ }}  &  \vtop{\hbox{\strut $c(G) = 2p^3$ } }\\
			\hline
			28 &	\vtop{\hbox{\strut $\phi_{2}(321)c$ } } & - & \vtop{\hbox{\strut $\lambda_{1} = 1_{\langle \alpha^p \rangle} \cdot \lambda_{\langle \gamma \rangle }$, } \hbox{\strut $\psi_2 =  \psi_{\langle \alpha G' \rangle } \cdot 1_{\langle \alpha_{1}G' \rangle} \cdot 1_{\langle \gamma G' \rangle} $ }}  &  \vtop{\hbox{\strut $c(G) = 2p^3$ } }\\
			\hline
			29 &	\vtop{\hbox{\strut $\phi_{2}(321)d$ } } & - & \vtop{\hbox{\strut $\lambda_{1} = 1_{\langle \alpha^p \rangle} \cdot \lambda_{\langle \gamma \rangle }$, } \hbox{\strut $\psi_2 =  \psi_{\langle \alpha G' \rangle } \cdot 1_{\langle \alpha_{1}G' \rangle} \cdot 1_{\langle \gamma G' \rangle} $ }}  &  \vtop{\hbox{\strut $c(G) = p^4+p^2$ } }\\
			\hline
			30 &	\vtop{\hbox{\strut $\phi_{2}(321)f$ } } & - & \vtop{\hbox{\strut $\lambda_{1} = 1_{\langle \alpha^p \rangle} \cdot \lambda_{\langle \alpha_{1}^p \rangle } \cdot \lambda_{\langle \alpha_{2} \rangle }$, } \hbox{\strut $\psi_2 =  \psi_{\langle \alpha G' \rangle } \cdot 1_{\langle \alpha_{1}G' \rangle} \cdot 1_{\langle \alpha_{2} G' \rangle} $ } \hbox{\strut $\psi_3 =  1_{\langle \alpha G' \rangle } \cdot \psi'_{\langle \alpha_{1}G' \rangle} \cdot 1_{\langle \alpha_{2} G' \rangle} $ }}  &  \vtop{\hbox{\strut $c(G) = p^3+2p^2$ } }\\
			\hline
			31 &	\vtop{\hbox{\strut $\phi_{2}(222)b$ } } & - & \vtop{\hbox{\strut $\lambda_{1} = 1_{\langle \alpha^p \rangle} \cdot 1_{\langle \alpha_{1}^p \rangle } \cdot \lambda_{\langle \gamma \rangle }$, } \hbox{\strut $\psi_2 =  \psi_{\langle \alpha G' \rangle } \cdot 1_{\langle \alpha_{1}G' \rangle} \cdot 1_{\langle \gamma G' \rangle} $ } \hbox{\strut $\psi_3 =  1_{\langle \alpha G' \rangle } \cdot \psi'_{\langle \alpha_{1}G' \rangle} \cdot 1_{\langle \gamma G' \rangle} $ }}  &  \vtop{\hbox{\strut $c(G) = p^3+2p^2$ } }\\
			\hline
			
		\end{longtable}
	\end{tiny}

	\begin{tiny}
		
		\begin{longtable}[c]{|l|l|l|l|l|}
			\caption{$c(G)$ of VZ $p$-groups of order $p^6$ belonging to $\Phi_5$ family \label{t:2}}\\
			
			\hline
%			\multicolumn{5}{| c |}{Begin of Table}\\
%			\hline
			Sl. no. &	Group $G$ & Group $N$ & Set $X$  &  $c(G)$ \\
			& & & &\\
			\hline
			\endfirsthead
			
			\hline
			\multicolumn{5}{|c|}{Continuation of Table \ref{t:2}}\\
			\hline
			Sl. no. &	Group $G$ & Group $N$ & $X$ &  $c(G)$ \\ 
			\hline
			\endhead
			
			\hline
			\endfoot
			
			\hline
%			\multicolumn{5}{| c |}{End of Table}\\
%			\hline
			\endlastfoot
			
			1 &	\vtop{\hbox{\strut $\phi_{5}(21^4)a$ } \hbox{\strut $= \phi_{5}(2111) \times C_p $ }} & \vtop{\hbox{\strut $ \phi_{5}(2111) $ }} & $\lambda_{1} = \lambda_{\langle \alpha_1^p \rangle}$ & \vtop{\hbox{\strut $c(G) = c(N) + c(C_p)$ } \hbox{\strut $=  p^3 + p$ }}  \\
			\hline
			2 &	\vtop{\hbox{\strut $\phi_{5}(1^6)$ } \hbox{\strut $= \phi_{5}(1^5) \times C_p $ }} & \vtop{\hbox{\strut $ \phi_{5}(1^5) $ }} & $\lambda_{1} = \lambda_{\langle \beta \rangle} $ & \vtop{\hbox{\strut $c(G) = c(N) + c(C_p)$ } \hbox{\strut $= p^3 + p$ }}  \\
			\hline
			3 &	\vtop{\hbox{\strut $\phi_{5}(3111)$ }} & - & $\lambda_{1} = \lambda_{\langle \alpha_1^p \rangle} $ & \vtop{\hbox{\strut $c(G) =  p^4 $ } }  \\
			\hline
			4 &	\vtop{\hbox{\strut $\phi_{5}(2211)a$ } } & - & \vtop{\hbox{\strut $\lambda_{1} = 1_{\langle \alpha_1^p \rangle} \cdot \lambda_{\langle \alpha_{2}^p \rangle }$, } \hbox{\strut $\psi_2 =  \psi_{\langle \alpha_1 G' \rangle } \cdot 1_{\langle \alpha_{2}G' \rangle}  \cdot 1_{\langle \alpha_{3}G' \rangle} \cdot 1_{\langle \alpha_{4}G' \rangle} $ }}  &  \vtop{\hbox{\strut $c(G) = p^3+p^2$ } }\\
			\hline
			5 &	\vtop{\hbox{\strut $\phi_{5}(2211)b$ } } & - & \vtop{\hbox{\strut $\lambda_{1} = 1_{\langle \alpha_1^p \rangle} \cdot \lambda_{\langle \alpha_{3}^p \rangle }$, } \hbox{\strut $\psi_2 =  \psi_{\langle \alpha_1 G' \rangle } \cdot 1_{\langle \alpha_{2}G' \rangle}  \cdot 1_{\langle \alpha_{3}G' \rangle} \cdot 1_{\langle \alpha_{4}G' \rangle} $ }}  &  \vtop{\hbox{\strut $c(G) = p^3+p^2$ } }\\
			\hline
			6 &	\vtop{\hbox{\strut $\phi_{5}(21^4)b$ }} & - & $\lambda_{1} = \lambda_{\langle \gamma \rangle} $ & \vtop{\hbox{\strut $c(G)  = p^4 $ } }  \\
			\hline
			7 &	\vtop{\hbox{\strut $\phi_{5}(21^4)c$ } } & - & \vtop{\hbox{\strut $\lambda_{1} = 1_{\langle \alpha_1^p \rangle} \cdot \lambda_{\langle \beta \rangle }$, } \hbox{\strut $\psi_2 =  \psi_{\langle \alpha_1 G' \rangle } \cdot 1_{\langle \alpha_{2}G' \rangle}  \cdot 1_{\langle \alpha_{3}G' \rangle} \cdot 1_{\langle \alpha_{4}G' \rangle} $ }}  &  \vtop{\hbox{\strut $c(G) = p^3+p^2$ } }\\
			\hline
			
		\end{longtable}
	\end{tiny}

\begin{theorem}\label{thm:VZpgroupp6} Suppose $G$ is VZ $p$-group of order $p^6$ (odd prime $p$). Then 
$c(G) = q(G) = p(G)$ and we have one of the following possibilities:
\begin{enumerate}
\item If $|Z(G)| = p^4$, then $c(G) = p^4+p, p^3+p^2+p, p^3+p^2, p^3+2p, 3p^2, 2p^2+2p, 2p^2+p, 3p^2+p, p^2+3p, p^5, p^4+p^2, 2p^3$ or  $p^3+2p^2$.
\item  If $|Z(G)| = p^2$, then $c(G) = 2p^3, p^3+p, p^4$ or $p^3+p^2$.
\end{enumerate}
\end{theorem}
\begin{proof}
The result follows from Corollary \ref{extra2}, Table \ref{t:1} and \ref{t:2}.
\end{proof}

	\section{Camina $p$-group of class 3}
{\bf Notation.}	Throughout this section, $G_{2} = G'$ denotes the commutator subgroup of $G$. The nilpotency class of a nilpotent group $G$ is the number $n$ such that $G_{n} \neq 1$ and $G_{n+1}=1,$ where $G_{2} = [G,G] = G'$ and $G_{i+1}=[G_{i}, G]$, for $i>2.$ Let $N$ be a normal subgroup of $G.$ Then we denote the set $\Irr(G) \setminus \Irr(G/N)$ by $\Irr(G|N)$.

We start by recalling some basic results that we will need later.	The pair $(G,N)$ is said to be a Camina pair if for every $g\in G \setminus N,~ gN \subseteq Cl_{G}(g),$ the conjugacy class of $g.$ When $N=G'$, the group $G$ is called a Camina group.
	The following lemma gives us equivalent conditions for a pair $(G,N)$ to be a Camina pair.
	\begin{lemma} [\cite{SM}, Lemma 3] \label{L8}
		Let $N$ be a normal subgroup of $G$ and let $g\in G \setminus N.$ Then the following statements are equivalent:
		\begin{itemize}
			\item [\rmfamily(i)] $\chi(g)=0$ for all $\chi \in Irr(G|N).$
			\item [\rmfamily(ii)] $gN \subseteq Cl_{G}(g),$ the conjugacy class of $g.$
		\end{itemize}
	\end{lemma}
It is easy to see that if $(G, N)$ is a Camina pair, then $Z(G) \leq  N \leq G{}'$. By Lemma \ref{L8}, it is clear that if $G$ is a Camina group, then for all $g\in G \setminus G{}'$, $\chi(g) = 0$  for all $\chi\in \nl(G)$. Next theorem gives us bound on nilpotency class of Camina $p$-group. 

\begin{theorem}\textnormal{\cite{DS}}\label{thm:caminanilpotency} If $G$ is a finite Camina $p$-group, then the nilpotency class of $G$ is at most $3$, i.e., $G_4 = \{1\}$.
\end{theorem}
 
Now, we describe group theoretical property of Camina $p$-group and give a brief information about all non-linear irreducible characters of $G$.

		\begin{lemma} [\cite{M1}, Corollary 2.3] \label{L2}
		Let $G$ be a $p$-group of nilpotency class $r$. If $(G, G_{k})$ is a Camina pair, then $G_{i}/G_{i+1}$ has exponent $p$ for $k-1 \leq i \leq r.$
	\end{lemma}

	\begin{lemma} [\cite{M1}, Theorem 5.2] \label{L3}
		Let $G$ be a Camina $p$-group of nilpotency class 3 and let $|G/G_{2}| = p^{m}, |G_{2}/G_{3}|=p^{n}.$ Then
		\begin{enumerate}
			\item [\rmfamily(i)] $(G, G_{3})$ is a Camina pair,
			\item [\rmfamily(ii)] $m = 2n$ and $n$ is even.
		\end{enumerate} 
	\end{lemma}
	
	\begin{lemma} [\cite{M1}, Corollary 5.3] \label{L7}
	If $G$ is a Camina $p$-group of nilpotency class 3, then $Z_{2}(G) = G_{2}$ and $Z(G) = G_{3}$, where $Z_2(G)/Z(G)=Z(G/Z(G))$.
\end{lemma}
\begin{remark}\label{remark:camina3} Suppose that nilpotency class of Camina $p$-group $G$ is $2$, i.e., $1 < G{}' \leq Z(G)$. Since G is a Camina group, each nonlinear irreducible character of $G$ vanishes outside $G{}'$. Therefore, $G{}' = Z(G)$ and hence $G$ becomes VZ $p$-group. So by Corollary \ref{C2},  we get $c(G)$. Now Suppose $G$ be a Camina $p$-group of nilpotency class 3. Then from Lemma \ref{L2} and Lemma \ref{L7}, $G/G', G'/Z(G)$ and $Z(G)$ are elementary abelian $p$-groups and $Z(G)\leq G'$. Further, from Lemma \ref{L3}, $(G, Z(G))$ is a Camina pair, $|G/G'|=p^{2n}, |G'/Z(G)|=p^{n}$ and $|G/Z(G)|=p^{3n},$ where $n$ is even. Then, it is proved in (\cite[Section 3, pp. 51-52]{PS}), that $\nl(G) = \Irr(G|Z(G)) \sqcup \nl(G/Z(G))$ as a disjoint union and $\cd(G) = \{1, p^{n}, p^{3n/2}\}$. 
\end{remark}	
\subsection{Proof of the Theorem \ref{T2}}
To prove Theorem \ref{T2}, first we will prove the following lemma.

\begin{lemma}\label{lem:caminaclass3}
Let $G$ be a Camina $p$-group of nilpotency class $3$. Let $X$ be a subset of $\Irr(G)$ such that 
		\begin{equation} \label{eq:camina}
		\bigcap_{\chi \in X} \ker (\chi) = 1
		\end{equation}
		and equation \eqref{eq:camina} does not hold for any proper subset of $X$.
		Then $X\subseteq \Irr(G|Z(G))$.
\end{lemma}
\begin{proof}  Since $(G, Z(G))$ is a Camina pair, from Lemma \ref{L8}, for all $\chi \in \Irr(G|Z(G)), ~ \chi(g) = 0,~\forall~ g\in G \setminus Z(G)$.  Let $\chi$ be an arbitrary irreducible character in $\Irr(G|Z(G)).$  Then $\chi\downarrow_{Z(G)} = \chi(1) \lambda,$ where $\lambda \in \lin(Z(G)),$ and 
\begin{equation}\label{eq:24}
\ker (\chi) = \ker (\lambda) < Z(G) < G',~ \forall~\chi \in \Irr(G|Z(G)).
\end{equation}
	Secondly, if $\psi \in \nl(G/Z(G)),$ then $Z(G) \leq \ker (\psi)$. 
Suppose $X$ is a collection of some irreducible characters of $G$ satisfying equation \eqref{eq:camina} such that equation \eqref{eq:camina} does not hold for any proper subset of $X$. We know that all characters in $X$ can not be linear. So, at least one irreducible character in $X$, say $\chi$, must be non-linear which either belongs to $\Irr(G|Z(G))$ or to $\nl(G/Z(G))$ (see Remark \ref{remark:camina3}). If $\chi \in \Irr(G|Z(G))$, then $\ker (\chi) < Z(G) < G' \leq \ker (\eta),$ where $\eta \in \lin(G)$. Then $\ker (\chi) \cap \ker (\eta) = \ker (\chi) \neq 1.$ Now, if $\chi \in \nl(G/Z(G)),$ then $Z(G)\leq \big(\ker (\chi) \cap \ker (\eta)\big)$, where $\eta \in \lin(G).$ Therefore, we can not put any linear characters in $X.$ Hence $X\subseteq \nl(G) = \Irr(G|Z(G)) \sqcup \nl(G/Z(G))$ (disjoint union). Further, since $X$ must satisfy equation \eqref{eq:24}, $X$ can not be a subset of $\nl(G/Z(G))$.
Now, suppose $X$ intersects with $\Irr(G|Z(G))$ and $\nl(G/Z(G))$. Let $\chi, \psi \in X$ such that $\chi \in \Irr(G|Z(G))$ and $\psi \in  \nl(G/Z(G))$. Then $\ker \chi < Z(G) \leq \ker \psi \Rightarrow \ker \chi \cap \ker \psi = \ker \chi$ and this contradicts equation \eqref{eq:24}. Therefore, $X$ must be a subset of $\Irr(G|Z(G))$.
\end{proof}

\noindent {\it Proof of Theorem \ref{T2}.} Suppose $d(Z(G))=r$.
Let $H_{i}, ~ 1\leq i \leq r,$ denote the $r$ cyclic components of $Z(G)$, i.e. $Z(G)\cong \prod_{i=1}^r H_i$, and let $\widehat{H_{i}} = H_{1} \times \cdots \times H_{i-1} \times H_{i+1} \times \cdots \times H_{r}$ for each $i$ ($ 1\leq i \leq r$). By Remark \ref{remark:camina3}, $Z(G)$ is elementary abelian $p$-group, i.e. $|H_i|=p$ for each $i$. From Lemma \ref{lem:caminaclass3}, it is clear that, to obtain $c(G)$, we have to consider irreducible characters from the set $\Irr(G|Z(G))$. Now consider the set $X=\{\chi_1,\ldots, \chi_r\}\subseteq \Irr(G|Z(G))$, where $\chi_{i}\downarrow_{Z(G)} = \chi_{i}(1)\lambda_{i}$. Here $\lambda_{i}\in \lin(Z(G))$ is a linear character of $Z(G)$ such that $\ker(\lambda_i)= \widehat{H_{i}}$. Then 
$\cap_{i=1}^{r}\ker (\lambda_{i}) = \cap_{i=1}^{r}\ker (\chi_{i}) = 1$. Moreover, no proper subset of $X$ satisfies equation \eqref{eq:camina}.
	It is easy to observe the following points.
\begin{enumerate}
\item For each $i$, $\Gamma (\chi_i)=\Gamma(\lambda_i)$ and $|\Gamma (\chi_i)|=|\Gamma(\lambda_i)|=p-1$.
\item For each $i$, $m(\chi_i)=\chi_i(1)m(\lambda_i)$ and $\chi_i(1)=p^{3n/2}$.
\item For each $i$, $m(\lambda_i)=\big|\min \big\{\sum_{\sigma \in \Gamma(\lambda_i)}\lambda_i^{\sigma}(g) ~|~ g\in G \big\}\big|=|-1|=1$. 
\end{enumerate} 
In fact, for any $\chi\in \Irr(G|Z(G))$, $m(\chi)=\chi(1)$ and all the above three observations holds.  Now consider the character $$\xi=\sum_{i=1}^r \theta_i,$$ where $\theta_i=\sum_{\sigma\in \Gamma(\chi_i)}\chi_i^{\sigma}= \chi_i(1)\bigg(\sum_{\sigma \in \Gamma(\lambda_i)}\lambda_i^{\sigma}\bigg)$. Then by the above observations, we have $m(\xi)=r\chi_i(1)=rp^{3n/2}$ and $\xi(1)=\sum_{i=1}^r |\Gamma(\chi_i)|\chi_i(1)=r(p-1)p^{3n/2}$.
%Therefore, 
%$$c(G) = \xi(1)+m(\xi)=r(p-1)p^{3n/2}+rp^{3n/2}=rp^{\frac{3n}{2}+1}.$$
 Thus, 
$\xi(1)+m(\xi)=r(p-1)p^{3n/2}+rp^{3n/2}=rp^{\frac{3n}{2}+1}.$
Note that the set $X = \{ \chi_{i} \}_{i=1}^{r}$ gives a set $Y = \{ \lambda_i \}_{i=1}^{r}$, which is a minimal faithful quasi-permutation representation of $Z(G)$. Therefore,
\[ c(G) = \xi(1) + m(\xi) = rp^{\frac{3n}{2}+1}. \]
Moreover, we know that $c(Z(G))=rp$. Hence we have $$c(G)=p^{\frac{3n}{2}}c(Z(G))=|G/Z(G)|^{1/2}c(Z(G)).$$
Further, from Theorem \ref{thm:pgroup}, for Camina $p$-group (odd prime $p$) we get $$c(G)= \mu(G) = q(G) = |G/Z(G)|^{1/2} c(Z(G)).$$ 
This proves the result.

\begin{corollary} Let $G$ and $H$ be two Camina $p$-groups of same order, and of same nilpotency class. Then $c(G)=c(H)$ if and only if $Z(G)$ is isomorphic to $Z(H)$.
 \end{corollary}
%	\noindent {\bf Remark 5: $\mu(G)$ and $c(G)$ of Camina groups of odd order }\\
%	Let $G$ be a Camina group of odd order. As $G$ is nilpotent, it is isomorphic to the direct product of its Sylow subgroups.  If $G$ is  a Camina group of class 2, then it  has exactly two character degrees, and if $G$ is a Camina group of class 3, then it has exactly three character degrees. In both the above cases, all the Sylow subgroups of $G$, except one, are abelian. Let $G \cong P_{1} \times P_{2} \times \cdots \times P_{n},$ where $P_{i}$ is a sylow $p_{i}-$subgroup of $G,$ for each $i,$ $ 1\leq i\leq n,$ where $P_{n}$ is the non-abelian Camina $p_{n}-$group and the rest are abelian Sylow subgroups. Then 
%	\begin{align*}
%	\mu(G) &= \mu(P_{1})+ \mu(P_{2}) + \cdots + \mu(P_{n}) \quad (\text{  \cite{DW}, Corollary 2}) \\
%	&= c(P_{1})+ c(P_{2})+ \cdots c(P_{n}) \quad (\text{  \cite{BG}, Theorem 3.2})  \\
%	& = c(G) \quad (\text{  \cite{MG}}).
%	\end{align*}
%	Thus, we finally get $\mu(G)$ and $c(G)$ by using Corollary \ref{C1} in case of Camina group of class 2, or by using Theorem \ref{T2} in case of Camina group of class 3. 
	
\section{Acknowledgements}
Ayush Udeep acknowledges University Grants Commission, Government of India. The corresponding author acknowledges SERB, Government of India for financial support through grant (MTR/2019/000118).

\end{document}